\documentclass{article}

\usepackage{mystyle}
\usepackage{wasysym}
\usepackage{mathtools}
\begin{document}
\title{The uniqueness of elementary embeddings}
\author{Gabriel Goldberg\\ Evans Hall\\ University Drive \\ Berkeley, CA 94720}
\maketitle
\section{Introduction}
Much of the theory of large cardinals beyond a measurable cardinal
concerns the structure of elementary embeddings of the universe
of sets into inner models. This paper seeks to answer the question of
whether the inner model uniquely determines the elementary embedding.

The question cannot be answered assuming ZFC alone: 
in unpublished work, exposited in \cref{SCHSection}, Woodin observed that it is consistent that
there are distinct normal ultrafilters with the same ultrapower.
He proved, however, that definable embeddings
of the universe into the same model must agree \textit{on the ordinals,}
and under a strong version of the \HOD\ Conjecture,
he proved the same result for arbitrary elementary embeddings.
Woodin conjectured that the result can be proved in second-order set theory (NBG) with the Axiom of Choice. 
The first theorem of this paper confirms his conjecture:
\begin{repthm}{thm:unique_ords}
    Any two elementary embeddings from the universe into the same inner model agree
    on the ordinals.
\end{repthm}

In \cref{ExtendibleSection}, we prove stronger uniqueness properties of elementary embeddings
assuming global large cardinal axioms. 
To avoid repeating the same hypothesis over and over, we introduce the following terminology:
\begin{defn}
    If \(\delta\) is an ordinal,
    we say \textit{the uniqueness of elementary embeddings holds above \(\delta\)}
    if for any inner model \(M\), there is at most one elementary embedding from the universe
    into \(M\) with critical point greater than \(\delta\).
\end{defn}
We say \textit{the uniqueness of elementary embeddings holds} if
it holds above \(0\). The uniqueness of elementary embeddings 
is formulated in the language of second-order
set theory.

It turns out that the uniqueness of elementary embeddings holds
above sufficiently large cardinals:
\begin{repthm}{thm:extendible_uniqueness}
    The uniqueness of elementary embeddings holds above the least extendible cardinal.
\end{repthm}
The hypothesis of \cref{thm:extendible_uniqueness}
seems to be optimal: for example, \cref{thm:superconsistency}
shows that it is consistent with a proper class of supercompact cardinals
that the uniqueness of elementary embeddings fails above every cardinal.

Our analysis also yields the uniqueness of elementary 
embeddings (above \(0\)) from other hypotheses. 
The independence
results of \cite{Fuchs,HamkinsReitz, Reitz} 
are often taken to show that Reitz's Ground Axiom has no consequences. 
The following theorem indicates that this may not 
be completely true:
\begin{repthm}{thm:ga}[Ground Axiom]
    If there is a proper class of strongly compact cardinals,
    then the uniqueness of elementary embeddings holds.
\end{repthm}
The conclusion of this result (that is, the outright uniqueness of elementary embeddings)
cannot be proved from any large cardinal axiom by the proof of \cref{thm:consistency}
and the L\'evy-Solovay theorem \cite{LevySolovay}. 
It is unclear whether the strongly compact cardinals 
are necessary, though: it seems unlikely, but
it could be that the Ground Axiom alone suffices to prove the result.
With this in mind, we conclude the section by proving a similar result from a different hypothesis.
\begin{repthm}{thm:strong_ga}[Ground Axiom]
    If there is a proper class of strong cardinals,
    then the uniqueness of ultrapower embeddings holds.
\end{repthm} 

In \cref{UASection}, we consider the situation under the Ultrapower Axiom (UA), 
which turns out to be quite simple:
\begin{repthm}{thm:ua}[UA]
    The uniqueness of elementary embeddings holds.
\end{repthm}

Finally, we use \cref{thm:extendible_uniqueness} 
to analyze a principle called the Weak Ultrapower Axiom (Weak UA) under the assumption of 
an extendible cardinal.
Weak UA states that any two ultrapowers of the universe of sets have
a common internal ultrapower. 
Before this work, we knew of no consequences of Weak UA.
Here we sketch the proofs of some results indicating 
that above an extendible cardinal, Weak UA is almost as powerful as UA: 
\begin{repthm}{thm:weak_ua_bound}[Weak UA]
    If \(\kappa\) is extendible, then \(V\) is a 
    generic extension of \(\HOD\) by a forcing in \(V_\kappa\).
\end{repthm}
The reader familiar with Vop\v{e}nka's theorem will note
that this is just a fancy way of saying that every set is
ordinal definable from some fixed parameter \(x\in V_\kappa\).
We also show that UA holds in HOD for
embeddings with critical point greater than or equal to \(\kappa\).
(See \cref{thm:hod_ua} for a precise statement.)
We do not know how to show this is true in \(V\)!
Combining this with some proofs from \cite{GCH} allows
us to prove the GCH above the first extendible cardinal under Weak UA.
\begin{repthm}{thm:gch}[Weak UA]
    If \(\kappa\) is extendible, then for all cardinals \(\lambda \geq \kappa\), \(2^\lambda = \lambda^+\).
\end{repthm}
\section{Preliminaries}
\subsection{Ultrapower embeddings and extender embeddings}
If \(P\) and \(Q\) are models of ZFC, \(i: P\to Q\) is an elementary embedding,
\(X\in P\) is a set, and \(a\in j(X)\), then 
there is a minimum elementary substructure of \(Q\) containing
\(i[P]\cup \{a\}\); namely, the substructure
\[H^Q(i[P]\cup \{a\}) = \{i(f)(a) : f\in P, f : X\to P\}\]
The fact that \(H^Q(i[P]\cup \{a\})\) is an elementary substructure of \(Q\) is a consequence of the Axiom of Choice
and the Axiom of Collection, applied in \(P\), 
and is really just a restatement of \L\'os's Theorem
applied to the \(P\)-ultrafilter \(U = \{A\subseteq X: A\in P,a\in i(A)\}\).

The same argument shows that for any \(X\in P\)
and set \(B\subseteq i(X)\),
\[H^Q(i[P]\cup B) = \{i(f)(a) : a\in B,f\in P, f : X\to P\}\]
is an elementary substructure of \(Q\).
\begin{defn}
    Suppose \(P\) and \(Q\) are models of set theory. An elementary embedding \(i : P\to Q\) is an 
    \textit{ultrapower embedding}
    of \(P\) if there is some \(X\in P\) and some \(a\in i(X)\) such that 
    \(Q = H^Q(i[P]\cup \{a\})\). 
\end{defn}
An elementary embedding \(i : P\to Q\)
is an ultrapower embedding if and only if there is a \(P\)-ultrafilter
\(U\) and an isomorphism \(k : \Ult(P,U)\to Q\) such that \(k\circ i = j_U\).
\begin{defn}
    An elementary embedding \(i : P \to Q\) is an \textit{extender embedding}
    if there is a set \(A\) such that \(Q = H^Q(i[P]\cup i(A))\).
\end{defn}
Again, an elementary embedding is an extender embedding if and only if it is
isomorphic to the ultrapower of \(P\) by a \(P\)-extender (using only functions in \(P\)).

A \textit{generator} of an elementary embedding \(i :P\to Q\) is an ordinal \(\xi\) of \(Q\)
such that \(\xi\notin H^Q(i[P]\cup \xi)\). By the wellordering theorem,
\(Q = H^Q(i[P]\cup \Ord^Q)\), and if \(Q\) is wellfounded,
it follows easily that \(Q = H^Q(i[P]\cup G)\)
where \(G\) is the class of generators of \(i\). Thus
an elementary embedding is an extender embedding if and only if its generators
form a set. 
\begin{defn}\label{defn:lambda}Given an elementary embedding \(i :P\to Q\) and a set \(a\in Q\),
we let \(\lambda_i(a)\) denote the least cardinality of a set \(X\)
such that \(a\in i(X)\).\end{defn}
Note that \(\lambda_i(a)\) may not be defined, either because there is
no \(X\) such that \(a\in i(X)\) or because there is no minimum cardinality
of such a set. We will be focused solely on cofinal elementary embeddings
of wellfounded models, in which case \(\lambda_i(a)\) will always be defined.

By the wellordering theorem, \(\lambda_i(a)\) is the least ordinal \(\lambda\) of \(P\) such that
\(a \in H^{Q}(i[P]\cup i(\lambda))\). 
This immediately implies:
\begin{lma}\label{lma:discontinuity}
    Suppose \(i : P\to Q\) is an elementary embedding, \(a\) is a point in \(Q\),
    and \(\lambda = \lambda_i(a)\). 
    If \(\lambda > 1\), then \(i(\lambda) \neq \sup i[\lambda]\).\qed
\end{lma}
\section{The uniqueness of embeddings on the ordinals}\label{SCHSection}
\subsection{Woodin's results}
\begin{thm}[Woodin]\label{thm:consistency}
    If it is consistent that there is a measurable cardinal,
    then it is consistent that the uniqueness of elementary embeddings fails.
    \begin{proof}
        Suppose \(\kappa\) is a measurable cardinal,
        and assume without loss of generality that \(2^\kappa = \kappa^+\).
        Let \(\langle \dot{\mathbb P}_{\alpha\beta} : \alpha \leq \beta < \kappa\rangle\)
        be the Easton support iteration where \(\dot {\mathbb P}_{\alpha\alpha}\) is trivial
        unless \(\alpha\) is an inaccessible non-Mahlo cardinal, in which case
        \(\dot {\mathbb P}_{\alpha\alpha} = \Add(\alpha,1)\).
        Thus \(\dot{\mathbb P}_{0\alpha}\) names a partial order in \(V\),
        which we denote by \(\mathbb P_{0\alpha}\).

        Let \(G\subseteq {\mathbb P}_{0\kappa}\)
        be a \(V\)-generic filter, let \(G_{0\alpha}\subseteq {\mathbb P}_{0\alpha}\)
        be the restriction of \(G\) to \({\mathbb P}_{0\alpha}\), let
        \(\mathbb P_{\alpha\beta} = (\mathbb P_{\alpha\beta})_{G_{0\alpha}}\),
        and let \(G_{\alpha\beta}\subseteq \mathbb P_{\alpha\beta}\) be the 
        \(V[G_{0\alpha}]\)-generic filter induced by \(G\).

        In \(V\), let \(U\) be a normal ultrafilter on \(\kappa\).
        We claim that in \(V[G]\), there are distinct normal ultrafilters \(U_0\) and \(U_1\)
        extending \(U\) such that \(\Ult(V[G],U_0) = \Ult(V[G],U_1)\).
        Let \(j : V\to M\) be the ultrapower of \(V\) by \(U\).
        Let 
        \[\langle \dot{\mathbb P}_{\alpha\beta} : \alpha \leq \beta < j(\kappa)\rangle
        = j(\langle \dot{\mathbb P}_{\alpha\beta} : \alpha \leq \beta < \kappa\rangle)\]

        Let \(\mathbb P_{\kappa, j(\kappa)} = (\dot {\mathbb P}_{\kappa, j(\kappa)})_{G}\).
        Then since \(M[G]\) is closed under \(\kappa\)-sequences in \(V[G]\),
        \(\mathbb P_{\kappa, j(\kappa)}\) is \({\leq}\kappa\)-closed.
        Moreover the set of maximal antichains of \(\mathbb P_{\kappa, j(\kappa)}\) 
        that belong to \(M[G]\) has cardinality \(\kappa^+\) in \(V[G]\).
        Therefore working in \(V[G]\), one can construct an \(M[G]\)-generic
        filter \(G_{\kappa,j(\kappa)}\subseteq \mathbb P_{\kappa, j(\kappa)}\).
        Note that \(j[G_{0\kappa}]\subseteq G_{0\kappa} * G_{\kappa,j(\kappa)}\).
        Letting \(H_0 = G_{0\kappa} * G_{\kappa,j(\kappa)}\), this implies that
        \(j\) extends to an elementary embedding \(j_0 : V[G]\to M[H_0]\)
        such that \(j_0(G) = H_0\).

        Notice that one obtains a second \(M[G]\)-generic filter 
        \(G_{\kappa,j(\kappa)}^*\subseteq \mathbb P_{\kappa,j(\kappa)}\)
        by flipping the bits of each component of \(G_{\kappa,j(\kappa)}\).
        Let \(H_1 = G_{0\kappa} * G_{\kappa,j(\kappa)}^*\).
        Obviously, \(M[H_0] = M[H_1]\).
        Moreover, \(j[G]\subseteq H_0\), so \(j\) extends to an elementary embedding 
        \(j_1 : V[G]\to M[H_1]\) such that \(j_1(G) = H_1\).
        Letting \(U_0\) and \(U_1\) be the normal ultrafilters of \(V[G]\) derived
        from \(j_0\) and \(j_1\) respectively,
        we have \(\Ult(V,U_0) = M[H_0] = M[H_1] = \Ult(V,U_1)\).
        Since \(j_0(G_0) \neq j_1(G_1)\), \(U_0\neq U_1\).

        In any case, \(j_0 : V[G]\to M[H_0]\) and \(j_1 : V[G]\to M[H_1]\)
        witness the failure of the uniqueness of elementary embeddings.
    \end{proof}
\end{thm}

\begin{thm}
    It is consistent that there exist distinct normal ultrafilters \(U_0\) and \(U_1\)
    with the same ultrapower \(M\) such that \(j_{U_0}(U_0) = j_{U_1}(U_1)\).
    \begin{proof}
        Let \(U\) be a normal ultrafilter 
        on a measurable cardinal \(\kappa\) and let \(j : V\to M\)
        denote its ultrapower. Assume \(2^\kappa = \kappa^+\). 
        Let \(\mathbb P\) be the Easton product
        \(\prod_{\delta\in A}\Add(\delta,1)\) where \(I\) is the set of
        inaccessible non-Mahlo cardinals less than \(\kappa\).

        Let \(\mathbb Q = j(\mathbb P)\) and let \(\mathbb Q/\mathbb P\)
        denote the product  \(\prod_{\delta\in j(I)\setminus \kappa}\Add(\delta,1)\)
        as computed in \(M\). Thus \(\mathbb Q \cong \mathbb P\times (\mathbb Q/\mathbb P)\).
        Since \(\kappa\notin j(I)\),  \(\mathbb Q/\mathbb P\) is \({\leq}\kappa\)-closed in \(M\), 
        and hence \(\mathbb Q/\mathbb P\)
        is \({\leq}\kappa\)-closed and in \(V\). Also \(\mathbb Q/\mathbb P\) is 
        \(j(\kappa)\)-cc in \(M\), and so one can enumerate the antichains
        \(\langle A_\alpha :\alpha < \kappa^+\rangle\) of \(\mathbb Q/\mathbb P\)
        that belong to \(M\) using that \(|j(\kappa)| =\kappa^+\).

        Fix a wellorder 
        \(\preceq\) of \(\mathbb Q/\mathbb P\), and let
        \(\langle p_\alpha : \alpha < \kappa^+\rangle\)
        be a continuous descending sequence in \(\mathbb Q/\mathbb P\)
        defined by letting \(p_{\alpha+1}\) be the \(\preceq\)-least
        element of \(\mathbb Q/\mathbb P\) below \(p_\alpha\)
        that belongs to \(A_\alpha\). Then 
        \(G = \{p\in \mathbb Q/\mathbb P :\exists \alpha\, p_\alpha\leq p\}\)
        is an \(M\)-generic filter.

        We denote by \(\sigma_\alpha\) the involution
        of \(\mathbb P\) that flips the bits of the Cohen sets
        added to cardinals above \(\alpha\). We overload notation
        by denoting the involution of \(\mathbb Q\)
        that flips the bits of the Cohen sets
        added to cardinals above \(\alpha\) in exactly the same way.

        Now we pass to a forcing extension: let \(H\subseteq \mathbb P\) be a \(V\)-generic filter.
        In \(V[H]\), we extend \(j\) in two different ways.
        Let \(j_0 : V[H]\to M[H\times G]\) be the unique extension of \(j\)
        such that \(j_0(H)= H\times G\). Let \(j_1 :V[H]\to M[H\times G]\)
        be the unique extension of \(j\) such that \(j_1(H) = \pi_\kappa(H\times G)\).

        Finally, a computation: 
        \begin{align*}j_1(j_1)(H\times G) &= j_1(j_1)(\sigma_\kappa(j_1(H)))\\
        &= \sigma_\kappa (j_1(j_1(H)))\\ 
        &= \sigma_\kappa (j_1(\sigma_\kappa(H\times G)))\\
        &= \sigma_\kappa\circ \sigma_{j(\kappa)} (j_1(H\times G))\\
        &= \sigma_\kappa\circ \sigma_{j(\kappa)} (\sigma_\kappa(H\times G)\times j_1(G))\\
        &= \sigma_\kappa\circ \sigma_{j(\kappa)} (\sigma_\kappa(H\times G)\times j_0(G))\\
        &= \sigma_\kappa \circ \sigma_\kappa(H\times G\times j_0(G))\\
        &= H\times G\times j_0(G)\\
        &= j_0(H)\times j_0(G)\\
        &= j_0(H\times G)\\
        &= j_0(j_0(H))\\
        &= j_0(j_0)(j_0(H))\\
        &= j_0(j_0)(H\times G)
        \end{align*}

        Since \(j_0(j_0)\) and \(j_1(j_1)\) agree on \(M\cup \{H\times G\}\),
        they agree on \(M[H\times G]\), as desired.
    \end{proof}
\end{thm}

Given this independence result, the following theorem is quite counterintuitive:
\begin{thm}[Woodin]\label{thm:HOD}
    Assume \(V = \HOD\). Then the uniqueness of elementary embeddings holds.
\end{thm}
It is worth pondering why one cannot refute this theorem by first forcing the failure of 
the uniqueness of elementary embeddings as in \cref{thm:consistency} and then 
forcing \(V = \HOD\) by some highly closed coding forcing.
For definable elementary embeddings, Woodin proved more:
\begin{thm}[Woodin]\label{thm:Woodin}
    Suppose \(j_0,j_1 : V\to M\) are definable elementary embeddings from the universe into the same inner model.
    Then for every ordinal \(\alpha\), \(j_0(\alpha) = j_1(\alpha)\).
    \begin{proof}
        Fix a number \(n\) (in the metatheory), and we will prove the theorem
        for \(\Sigma_n\)-definable elementary embeddings. Towards a contradiction,
        let \(\alpha\) be the least ordinal such that there exist \(\Sigma_n\)-definable
        elementary embeddings \(j_0,j_1 : V\to M\) such that \(j_0(\alpha) \neq j_1(\alpha)\).
        Notice that \(\alpha\) is definable in \(V\) without parameters. Therefore
        if \(k_0,k_1 :V\to N\) are elementary embeddings, then \(k_0(\alpha)\)
        is the unique ordinal defined in \(N\) by the formula defining \(\alpha\) in \(V\),
        and similarly for \(k_1(\alpha)\). Hence \(k_0(\alpha) = k_1(\alpha)\), contradicting
        the definition of \(\alpha\).
    \end{proof}
\end{thm}
\subsection{Uniqueness of embeddings on the ordinals}
\cref{thm:Woodin} raises an interesting second-order question. Working in 
second-order set theory, suppose \(j_0,j_1 : V\to M\) are elementary embeddings. 
Must \(j_0\) and \(j_1\) agree on the ordinals?
Woodin conjectured that the answer is yes. Here we verify his conjecture.

\begin{thm}\label{thm:unique_ords}
    Any two embeddings from the universe of sets into the same inner model
    agree on the ordinals.
\end{thm}

Roughly speaking, we proceed by reducing the question
to the case of definable embeddings (in fact, ultrapower embeddings).

\begin{defn}
    An elementary embedding \(j : V\to M\) is \textit{almost an ultrapower embedding}
    if for every set \(B\subseteq M\), there is some \(a\in M\) such that
    \(B\subseteq H^M(j[V]\cup \{a\})\).
\end{defn}
I am grateful to Moti Gitik for pointing out an error in
the proof of this theorem that appeared in 
an early draft of this paper, which
has now been corrected:
\begin{thm}\label{thm:ultrapowers}
    Suppose \(j_0,j_1 : V\to M\) are elementary embeddings.
    Then there exist elementary embeddings \(i_0,i_1 : V\to N\)
    and an elementary embedding \(k : N\to M\)
    such that \(i_0\) and \(i_1\) are almost ultrapower embeddings
    and \(j_0 = k\circ i_0\) and \(j_1 = k\circ i_1\).
    \begin{proof}
        Suppose \(j_0,j_1 : V\to M\) are elementary embeddings.
        Let \(X = H^M(j_0[V]\cup j_1[V])\),
        let \(N\) be the transitive collapse of \(X\),
        and let \(k : N\to M\) be the inverse of the transitive collapse map.
        Let \(i_0,i_1 : V\to N\) be the collapses of \(j_0,j_1\);
        that is \(i_0 = k^{-1}\circ j_0\) and \(i_1 = k^{-1}\circ j_1\).
        
        We claim that for all sets \(A\), there is a point \(g\in i_1[V]\)
        such that \(i_1[A]\subseteq H^N(i_0[V]\cup \{g\})\).
        To see this, let \(B\) be a set
        of cardinality \(|A|\) such that \(i_0\restriction B = i_1\restriction B\).
        Note that such a set exists because \(i_0\) and \(i_1\)
        have an \(\omega\)-closed unbounded class of common fixed points.
        Let \(f : B\to A\)
        be a surjection.
        Let \(g = i_1(f)\). 
        For all \(a\in A\), \(a = f(b)\) for some
        \(b \in B\), and so 
        \[i_1(a) = i_1(f)(i_1(b)) = g(i_0(b)) \in H^N(i_0[V]\cup \{g\})\]
        Thus \(i_1[A]\subseteq H^N(i_0[V]\cup \{g\})\).

        We now show that \(i_0\) is almost an ultrapower embedding. Fix \(B\subseteq N\).
        Since \(N = H^N(i_0[V]\cup i_1[V])\),
        there is a set \(A\) such that \(B\subseteq H^N(i_0[V]\cup i_1[A])\).
        The previous paragraph yields \(g\in i_1[V]\) such that
        \(i_1[A]\subseteq H^N(i_0[V]\cup \{g\})\). Hence
        \(B\subseteq H^N(i_0[V]\cup \{g\})\), as desired.
    \end{proof}
\end{thm}

Under favorable cardinal arithmetic hypotheses, one can remove the word ``almost''
in the statement of the previous theorem.
We will say here that the \textit{eventual singular cardinals hypothesis} (eventual SCH) holds
if for all sufficiently large strong limit cardinals
\(\lambda\) of cofinality \(\omega\), \(2^\lambda = \lambda^+\).
(By Silver's theorem, this also implies \(2^\lambda = \lambda^+\) for
strong limit singular cardinals of uncountable cofinality, but we will not need this,
and this form will be more convenient. Our eventual SCH is a bit weaker than
the more natural version asserting that for all sufficiently large
singular \(\lambda\), \(\lambda^{\cf(\lambda)} = 2^{\cf(\lambda)}\cdot \lambda^+\).)
\begin{lma}[Eventual SCH]\label{lma:extender}
    Any elementary embedding from the universe into an inner model
    closed under \(\omega\)-sequences is an extender embedding.
    \begin{proof}
        Suppose not, and fix an elementary embedding \(j : V\to M\) 
        such that \(M\) is closed under \(\omega\)-sequences but \(j\) is not
        an extender embedding.
        Then the class \(\{\lambda_j(a) : a\in M\}\) is unbounded. (See \cref{defn:lambda}.)
        Otherwise, let \(\gamma\) be its supremum. Then
        \[M = H^M(j[V]\cup j(\gamma))\] 
        contrary to the fact that
        \(j\) is not an extender embedding.

        Let \(\eta\) be such that for all strong limit cardinals
        \(\lambda > \eta\) of countable cofinality, \(2^\lambda = \lambda^+\).
        By recursion, construct a sequence of points \(a_n\in M\)
        such that \(\lambda_j(a_0) > \eta\), 
        \(\lambda_j(a_{n+1}) > j(2^{\lambda_j(a_n)})\).
        Let \(\lambda = \sup_{n < \omega} \lambda_j(a_n)\).
        
        Note that \(\lambda\) is a strong limit cardinal
        of countable cofinality, so \(2^\lambda = \lambda^+\). 
        Moreover, \(j[\lambda]\subseteq\lambda\),
        and so since \(j(\omega) = \omega\), \(j(\lambda) = \lambda\).
        In particular, \(j\) is continuous at \(\lambda\).
        Also \(j(\lambda^+) = (j(\lambda)^+)^M \leq \lambda^+\).
        In particular, \(j\) is continuous at \(\lambda^+\).
        
        Let \(a = \langle a_n : n < \omega\rangle\).
        Then \(a\in M\). In fact, since \(a\subseteq H(j[V]\cup j(\lambda))\), 
        \(a\in H(j[V]\cup j({}^\omega\lambda))\).
        So \(\lambda_j(a)\leq \lambda^\omega = \lambda^+\).
        On the other hand since \(\langle \lambda_j(a_n) : n < \omega\rangle\)
        is cofinal in \(\lambda\), \(\lambda_j(a)\geq \lambda\).
        Thus \(\lambda_j(a)\) is either \(\lambda\) or \(\lambda^+\).
        Since \(j\) is continuous at \(\lambda\) and \(\lambda^+\),
        this contradicts \cref{lma:discontinuity}.
    \end{proof}
\end{lma}

\begin{cor}[Eventual SCH]\label{cor:almost_ultra}If \(j : V\to M\) is almost an ultrapower embedding,
    then \(j\) is an ultrapower embedding.
    \begin{proof}
        It is easy to see that \(M\) is closed under \(\omega\)-sequences,
        and therefore \cref{lma:extender} implies that \(j\) is an extender embedding.
        Essentially by definition, an extender embedding that is almost an ultrapower embedding
        is indeed an ultrapower embedding.
    \end{proof}
\end{cor}

\begin{cor}[Eventual SCH]\label{cor:SCH}If \(j_0,j_1 : V\to M\) are
    elementary embeddings, then \(j_0\restriction \Ord = j_1\restriction \Ord\).
    \begin{proof}
        Apply \cref{thm:ultrapowers} to reduce to the case that \(j_0\) and \(j_1\)
        are almost ultrapowers. Then apply \cref{cor:almost_ultra}
        to conclude that they are in fact ultrapower embeddings.
        Finally, since ultrapower embeddings are definable, apply Woodin's
        theorem (\cref{thm:Woodin}) to conclude the corollary.
    \end{proof}
\end{cor}

We now turn to the proof of the uniqueness of elementary embeddings on the ordinals
without SCH, 
for which it is convenient to introduce the notion of 
the \textit{tightness function} of an elementary embedding.
\begin{defn}
    Suppose \(j : V\to M\) is an elementary embedding and \(X\) is a set.
    Then \(\tr{j}{X}\) denotes the minimum \(M\)-cardinality
    of a set \(A\in M\) such that \(j[X]\subseteq A\).
\end{defn}
The tightness function turns out to depend only on the cardinality of its argument:
\begin{lma}\label{lma:j_cover}
    Suppose \(j : V\to M\) is an elementary embedding.
    If \(|X|\leq |Y|\), then \(\tr{j}{X} \leq \tr{j}{Y}\).
    \begin{proof}
        Let \(f : Y\to X\) be a surjection. For any \(A\in M\), if
        \(j[Y]\subseteq A\), then \(j[X]\subseteq j(f)[A]\). As a consequence 
        \(\tr{j}{X} \leq \tr{j}{Y}\).
    \end{proof}
\end{lma}
We therefore will focus on \(\tr{j}{\lambda}\) where \(\lambda\) is a cardinal.

We want to get into the situation where we can apply \cref{cor:almost_ultra},
and for this we need Solovay's argument proving SCH above a strongly compact cardinal.
\begin{lma}[Solovay]\label{lma:solovay_SCH}
    Suppose \(\lambda\) is a singular strong limit cardinal of countable
    cofinality and there is an elementary embedding \(j : V\to M\)
    such that \(j\) is discontinuous at \(\lambda^+\). Then \(2^\lambda = \lambda^+\).
    \begin{proof}[Sketch]
        We may assume that \(j\) is the ultrapower of the universe by an ultrafilter on \(\lambda^+\).
        Note that \(\tr{j}{\lambda^+} = \cf^M(\sup j[\lambda^+]) < j(\lambda)\).
        Also \(\tr{j}{\lambda^\omega} = (\tr{j}{\lambda})^\omega 
        \leq (\tr{j}{\lambda^+})^\omega < j(\lambda)\).
        Assume towards a contradiction that \(\lambda^\omega > \lambda^+\).
        Then \(\tr{j}{\lambda^{++}} \leq \tr{j}{\lambda^\omega}< j(\lambda)\).
        But this implies \(j\) is discontinuous at \(\lambda^{++}\), which contradicts that 
        \(j\) is the ultrapower of the universe by an ultrafilter on \(\lambda^+\).
    \end{proof}
\end{lma}
A more complete proof appears in \cite[Lemma 7.2.18]{UA}.

\begin{proof}[Proof of \cref{thm:unique_ords}]
        We start with a simple observation.
        Suppose \(\delta\) is a regular cardinal such that \(j_0(\delta) < j_1(\delta)\).
        Then \(j_1\) is discontinuous at \(\delta\).
        To see this, let \(X\) be a set of common fixed points of \(j_0\) and \(j_1\) such that
        \(|X| = \delta\).
        Then \(j_1[X] = j_0[X]\) is covered by \(j_0(X)\), which has size \(j_0(\delta)\) in \(M\).
        By \cref{lma:j_cover}, 
        \(j_1[\delta]\) is covered by a set \(B\in M\) such that \(|B|^M = j_0(\delta)\).
        It follows that \(\sup j_1[\delta] \neq j_1(\delta)\): otherwise \(\cf^M(j_1(\delta)) 
        = j_0(\delta)\), contradicting that \(\delta\) is regular.

        Assume towards a contradiction that
        there is an ordinal \(\alpha\) such that \(j_0(\alpha) \neq j_1(\alpha)\).
        Without loss of generality, assume \(j_0(\alpha) < j_1(\alpha)\).

        Assume towards a contradiction that for cofinally many strong limit cardinals \(\lambda\)
        of countable cofinality, \(2^\lambda > \lambda^+\).
        Let \(\lambda\) be the \(\alpha\)-th strong limit cardinal of countable cofinality
        for which \(2^\lambda > \lambda^+\). Then \(j_0(\lambda)\)
        is the \(j_0(\alpha)\)-th such cardinal in \(M\), and \(j_1(\lambda)\)
        is the \(j_1(\alpha)\)-th. Hence \(j_0(\lambda) < j_1(\lambda)\). As a consequence,
        \(j_0(\lambda^+) < j_1(\lambda^+)\). So \(j_1\) is discontinuous
        at \(\lambda^+\) by the claim. It therefore follows by \cref{lma:solovay_SCH} that
        \(2^\lambda = \lambda^+\), which is a contradiction.

        Applying \cref{cor:SCH}, \(j_0\restriction \Ord  = j_1\restriction \Ord\),
        contrary to assumption. 
    \end{proof}
\section{The uniqueness of embeddings above large cardinals}\label{ExtendibleSection}
Intuitively, an ultrafilter \(U\) on a set \(X\) is a ``generalized element'' of \(X\).
In this section, we study the generalization of ordinal definability that arises
from this intuition: namely, definability from ultrafilters on ordinals.
Since it turns out that \textit{every} set is definable from an ultrafilter
on an ordinal,
it is natural in the context of large cardinals to study
the sets definable from increasingly complete such ultrafilters. After all,
the ordinal definable sets are precisely the sets definable from \textit{principal}
ultrafilters on ordinals, or in other words, from ultrafilters that are \(\kappa\)-complete
for all cardinals \(\kappa\). The analysis of this concept leads to a proof
of the uniqueness of elementary embeddings above an extendible cardinal.
\subsection{Completely definable sets}
Suppose \(\kappa\) is an infinite cardinal. 
A set is \textit{\(\kappa\)-completely definable} if it is definable in the structure \((V,\in)\)
from a \(\kappa\)-complete ultrafilter on an ordinal. A set is \textit{completely
definable} if it is \(\delta\)-completely definable for all infinite cardinals \(\delta\).
The class of \(\kappa\)-completely definable sets is denoted by \(\CD(\kappa)\)
and the class of completely definable sets by \(\CD\).

We mentioned that every set is definable from an ultrafilter on an ordinal,
so one might expect that under large cardinal axioms,
for example if \(\kappa\) is strongly compact, then 
every set is definable from a \(\kappa\)-complete
ultrafilter on an ordinal. But in fact, no matter what large
cardinal axioms one assumes, it is consistent that there is
a set that is \textit{not} \(\omega_1\)-completely definable. This is because if \(g\)
is Cohen generic over \(V\), then \(g\) is not \(\omega_1\)-completely definable in \(V[g]\).
Yet all known large cardinal axioms are upwards absolute from \(V\) to \(V[g]\).

For any set \(X\), let \(\UF_\kappa(X)\)
be the set of \(\kappa\)-complete ultrafilters on \(X\). 
Let \(\UF_\kappa(\Ord) = \bigcup_{\delta\in \Ord} \UF_\kappa(\delta)\).
Note that any ordinal can be coded by a principal
ultrafilter on an ordinal and any sequence of ultrafilters on ordinals can be coded
by a single ultrafilter on an ordinal. As an immediate consequence,
we obtain a more familiar characterization of \(\CD(\kappa)\):
\begin{prp}\label{prp:cd_od}
    For any cardinal \(\kappa\), \(\CD(\kappa) = \OD_{\UF_\kappa(\Ord)}\).\qed
\end{prp}

In a somewhat artificial sense, 
complete definability is just a quantifier-flip away from ordinal definability:
\(x\) is ordinal definable if \(x\) is definable from an ultrafilter on an ordinal 
that is \(\kappa\)-complete for all cardinals \(\kappa\);
\(x\) is completely definable if for all cardinal \(\kappa\), \(x\) is definable from 
a \(\kappa\)-complete ultrafilter on an ordinal.

A \(\kappa\)-completely definable set \(x\) is \textit{hereditarily \(\kappa\)-completely definable} 
(resp.\ \textit{hereditarily completely definable}) if
every element of its transitive closure is also \(\kappa\)-completely definable
(resp.\ completely definable).
Thus the class \(\HCD(\kappa)\) of all hereditarily \(\kappa\)-completely definable sets
is the largest transitive subclass of \(\CD(\kappa)\),
and the class \(\HCD\) of all hereditarily completely definable sets is the
largest transitive subclass of \(\CD\).

\begin{prp}\label{prp:hcd_hod}
    For any cardinal \(\kappa\), \(\HCD(\kappa)\) is an inner model of \ZF.
    In fact, \(\HCD(\kappa) = \HOD_{\UF_\kappa(\Ord)}\).
    \begin{proof}
        That \(\HCD(\kappa) = \HOD_{\UF_\kappa(\Ord)}\) is immediate by \cref{prp:cd_od}.
        The structure
        \(\HOD_{\UF_\kappa(\Ord)}\) 
        is a model of \ZF\ since \(\UF_\kappa(\Ord)\) is itself definable from
        an ordinal. 
    \end{proof}
\end{prp}

Let \(\kappa_\alpha\) denote the supremum of the first \(\alpha\) measurable cardinals.
We have a decreasing sequence of inner models:
\[V = \HCD(\omega)\supseteq \HCD(\omega_1) = \HCD(\kappa_1)\supseteq \HCD(\kappa_2)\supseteq\cdots\supseteq
\HCD(\kappa_\alpha)\supseteq\cdots \supseteq\HCD\supseteq \HOD\]

One reason the \(\kappa\)-completely definable sets are interesting is that for certain
large cardinals \(\kappa\), \(\HCD(\kappa)\) is a model of \ZFC. 
\begin{thm}\label{thm:hcd_zfc}
    If \(\kappa\) is a strongly compact cardinal, then \(\HCD(\kappa)\) is a model
    of \ZFC.
\end{thm}

For this we will use
the following facts.
\begin{lma}\label{lma:hcd_ultrapower}
    A set \(x\) is \(\kappa\)-completely definable if and only if
    there is an ultrapower embedding \(j : V\to M\) with \(\crit(j) \geq \kappa\)
    such that \(x\) is definable in the structure  \((V,\in,j)\) from ordinal parameters.
    \begin{proof}
        For the forwards direction, note that any \(\kappa\)-complete ultrafilter \(W\)
        on an ordinal is definable in the structure \((V,\in,j_W)\) from the ordinal \(\id_W\);
        hence any set definable in \(V\) from \(W\) is definable from ordinal parameters 
        in the structure \((V,\in,j_W)\).

        For the converse, note that if \(j : V\to M\) is an 
        ultrapower embedding with \(\crit(j) \geq \kappa\),
        then (by the wellordering theorem)
        there is a \(\kappa\)-complete ultrafilter \(W\) on an ordinal such that \(j = j_W\).
    \end{proof}
\end{lma}
\begin{thm}[Kunen]\label{thm:kunen}
    Suppose \(\mathcal U\) is a fine ultrafilter on \(P_\kappa(P(\delta))\)
    and \(W\) is a \(\kappa\)-complete ultrafilter on \(\delta\). Then there is some
    \(\alpha < j_\mathcal U(\delta)\) such that 
    \(W = \{A\subseteq \delta : M_\mathcal U\vDash \alpha\in j_{\mathcal U}(A)\}\).
    \begin{proof}
        Let \(\sigma = \id_{\mathcal U}\). 
        Since \(\mathcal U\) is a fine ultrafilter on \(P_\kappa(P(\delta))\),
        \(j_\mathcal U[P(\delta)]\subseteq \sigma\subseteq j_\mathcal U(P(\delta))\)
        and \(|\sigma|^{M_\mathcal U} < j_\mathcal U(\kappa)\). 
        Let \(B = j_\mathcal U(W)\cap \sigma\),
        so that \(B\in M_\mathcal U\), \(j_\mathcal U[W]\subseteq B \subseteq j_\mathcal U(W)\),
        and \(|B|^{M_\mathcal U} < j_\mathcal U(\kappa)\). 
        Since \(W\) is \(\kappa\)-complete, \(j_\mathcal U(W)\) is \(j_\mathcal U(\kappa)\)-complete,
        and hence \(\bigcap B\in j_\mathcal U(W)\), and in particular, there is some
        \(\alpha\in \bigcap B\). Using that \(j_\mathcal U[W]\subseteq B\), it is easy to see that 
        \(W\subseteq \{A \subseteq \delta : M_\mathcal U\vDash \alpha\in j_\mathcal U(A)\}\),
        and so by the maximality of \(W\), equality holds.
    \end{proof}
\end{thm}
\begin{cor}\label{cor:uf_wo}
    Suppose \(\kappa\leq \delta\) are cardinals 
    and \(\kappa\) is \(2^\delta\)-strongly compact. Then there is a \(\kappa\)-completely
    definable wellorder of \(\UF_\kappa(\delta)\).
    \begin{proof}
        Since \(\kappa\) is \(2^\delta\)-strongly compact,
        there is a \(\kappa\)-complete fine ultrafilter \(\mathcal U\)
        on \(P_\kappa(P(\delta))\). \cref{thm:kunen} permits us to
        define a function \(g: \UF_\kappa(\delta)\to j_\mathcal U(\delta)\)
        by setting \(g(W)\) equal to the least \(\alpha < j_\mathcal U(\delta)\)
        such that \(W = \{A\subseteq \delta : \alpha \in j_\mathcal U(A)\}\).
        Then \(g\) is an injection and \(g\) is \(\kappa\)-completely definable
        by \cref{lma:hcd_ultrapower}. Set \(W_0\preceq W_1\) if \(g(W_0)\preceq g(W_1)\).
        Then \(\preceq\) is a wellorder of \(\UF_\kappa(\delta)\)
        since it order-embeds into the wellorder \(j_\mathcal U(\delta)\), and \(\preceq\) is
        \(\kappa\)-completely definable since \(\preceq\) is definable from \(g\).
    \end{proof}
\end{cor}

\begin{proof}[Proof of \cref{thm:hcd_zfc}]
        By \cref{cor:uf_wo}, for any ordinal \(\delta\),
        \(\UF_\kappa(\delta)\) admits a \(\kappa\)-completely definable wellorder.
        As a consequence, the class \(\OD_{\UF_\kappa(\delta)}\) admits a \(\kappa\)-completely
        definable wellorder. By \cref{prp:hcd_hod},
        \(\CD(\kappa) = \OD_{\UF_\kappa(\Ord)} = \bigcup_{\delta\in \Ord}\OD_{\UF_\kappa(\delta)}\).

        Now fix an ordinal \(\alpha\), and we will show that there is a
        wellorder of \(\HCD(\kappa)\cap V_\alpha\) in
        \(\HCD(\kappa)\).
        Since \(\CD(\kappa)\) is the increasing union of the classes \(\OD_{\UF_\kappa(\delta)}\), 
        the pigeonhole principle implies that
        for any ordinal \(\alpha\), \(\CD(\kappa)\cap V_\alpha = \OD_{\UF_\kappa(\delta)}\cap V_\alpha\)
        for some cardinal \(\delta\geq \kappa\).
        The restriction of any \(\kappa\)-completely definable wellorder
        of \(\OD_{\UF_\kappa(\delta)}\) to \(\CD(\kappa)\cap V_\alpha\) yields
        a \(\kappa\)-completely definable wellorder of the latter set. Restricting further, 
        \(\HCD(\kappa)\cap V_\alpha\subseteq \CD(\kappa)\cap V_\alpha\) admits a \(\kappa\)-completely
        definable wellorder. This wellorder is trivially hereditarily \(\kappa\)-completely
        definable (its transitive closure is equal to \(\HCD(\kappa)\cap V_\alpha\),
        at least if \(\alpha\) is a limit ordinal), and there is a
        wellorder of \(\HCD(\kappa)\cap V_\alpha\) in
        \(\HCD(\kappa)\), as claimed.

        Since \(\HCD(\kappa)\) satisfies that for any \(\alpha\), \(\HCD(\kappa)\cap V_\alpha\) 
        is wellorderable, \(\HCD(\kappa)\) satisfies the Axiom of Choice.
\end{proof}

Note that while \(\HCD(\kappa)\) is an inner model of ZFC whenever \(\kappa\) is strongly compact,
it is not provable in ZFC that the entire class \(\HCD(\kappa)\) 
can be definably wellordered from any parameter whatsoever. 
(Indeed, by \cref{thm:hcd_ground}, this holds if and only if \(V\) itself can be definably 
wellordered from a parameter.)

We now show that when \(\kappa\) is strongly compact, \(\HCD(\kappa)\) is a very large model.
In fact, \(\HCD(\kappa)\) is a ground of the universe, in the sense of set theoretic geology.
Recall that if \(N\subseteq M\) are models of set theory, \(N\) is said to be a \textit{ground of \(M\)}
if there is a partial order \(\mathbb P\in N\) and an \(N\)-generic filter \(G\subseteq \mathbb P\) in
\(M\) such that \(M = N[G]\).

For any set \(x\), let \(\CD(\kappa)_x\) denote the class of sets that 
are \(\kappa\)-completely definable from \(x\), and let \(\HCD(\kappa)_x\) denote
the class of all sets hereditarily \(\kappa\)-completely definable from \(x\).

\begin{prp}\label{prp:v_hcd_x}
    If \(\kappa\) is strongly compact and \(x\) is
    a set such that \(V_\kappa\subseteq \CD(\kappa)_x\),
    then \(V = \HCD(\kappa)_x\).
    \begin{proof}
        We first claim that for every cardinal \(\lambda\) such that
        \(\lambda^{<\kappa} = \lambda\), there is a \(\kappa\)-independent family of subsets
        of \(\lambda\) of cardinality \(\lambda\) that belongs to \(\HCD(\kappa)_x\).
        To see this, let \(j : V\to M\) be an ultrapower embedding with \(\crit(j) = \kappa\) 
        and \(j(\kappa) > \lambda\). Then \(M\) is closed under \(\kappa\)-sequences.
        We claim that \(V_{j(\kappa)}\cap M\subseteq \HCD(\kappa)_x\). By elementarity, 
        \[V_{j(\kappa)}\cap M \subseteq j(V_{\kappa})\subseteq j(\CD(\kappa)_x)  = \CD^M(j(\kappa))_{j(x)}
        \subseteq \CD^M(\kappa)_{j(x)}\]
        So it suffices to show that \(\CD^M(\kappa)_{j(x)}\subseteq \CD(\kappa)_{x}\).
        
        Nte that \(M\) 
        and \(j(x)\) are definable over the structure \((V,\in,j)\) from \(x\).
        Also every ultrapower embedding \(i : M\to N\) with \(\crit(i)\geq \kappa\)
        is definable over \((V,\in,j,k)\) from ordinal parameters 
        for some ultrapower embedding \(k :V\to N\)
        with \(\crit(k)\geq \kappa\). For this, take \(k = i\circ j\), let
        \(\alpha\in \Ord\) be a seed of \(j\) (so \(M = H^M(j[V]\cup \{\alpha\})\)),
        and let \(\beta = i(\alpha)\).
        Then given \(a\in M\), \(i(a)\) can be computed by
        choosing any \(f\in V\) such that \(a = j(f)(\alpha)\) 
        and noting that 
        \[i(a) = i(j(f)(\alpha)) = k(f)(\beta)\]
        This defines \(i\) in the structure 
        \((V,\in,j,k)\) from the ordinals \(\alpha\) and \(\beta\).

        Since \(M\), \(j(x)\), and every ultrapower embedding of \(M\)
        are each ordinal definable from \(x\) over a structure of the form \((V,\in,j,k)\) 
        where \(j\) and \(k\) are ultrapower embeddings with critical point at least \(\kappa\),
        a slight generalization of \cref{lma:hcd_ultrapower} yields \(\CD^M(\kappa)_{j(x)}\subseteq 
        \CD(\kappa)_{x}\).

        Since \(M\) is a model of \ZFC\ that is
        closed under \(\kappa\)-sequences, there is a \(\kappa\)-independent family
        of subsets of \(\lambda\) in \(M\). Since \(V_{j(\kappa)}\cap M\subseteq \HCD(\kappa)_x\),
        and \(\lambda < j(\kappa)\), this \(\kappa\)-independent family belongs to \(\HCD(\kappa)_x\).
        
        Finally, fix a set of ordinals \(A\). We will prove that
        \(A\in \HCD(\kappa)_x\). Fix \(\lambda\geq \sup A\) such that \(\lambda^{<\kappa} = \lambda\), 
        and let \(\langle S_\alpha : \alpha < \lambda\rangle\in \HCD(\kappa)_x\) be a
        \(\kappa\)-independent family of subsets of \(\lambda\).
        Let \(F\) be the \(\kappa\)-complete filter on \(\lambda\) generated by 
        \[B = \{S_\alpha :\alpha\in A\}\cup \{\lambda\setminus A : \alpha\notin A\}\]
        Since \(\kappa\) is strongly compact,
        there is a \(\kappa\)-complete ultrafilter \(U\) extending \(F\).
        Let \(W = U\cap \HCD(\kappa)_x\).
        Since \(U\) is a \(\kappa\)-complete ultrafilter on an ordinal,
        \(W \in \HCD(\kappa)_x\). Now \(A = \{\alpha : S_\alpha\in W\}\), and so
        \(A\in \HCD(\kappa)_x\).

        Since we have shown that every set of ordinals belongs to \(\HCD(\kappa)_x\),
        the Axiom of Choice yields that \(V = \HCD(\kappa)_x\).
    \end{proof}
\end{prp}

\begin{thm}[Vopenka]\label{thm:vopenka}
    Suppose \(C\) is a class such that for all ordinals \(\alpha\), 
    there is a wellorder of \(C\cap V_\alpha\) in \(\OD_C\). 
    Then for any set of ordinals \(x\), \(\HOD_C\) is a ground of 
    \(\HOD_{C\cup \{x\}}\).
    \begin{proof}[Sketch]
        Suppose \(x\subseteq \beta\).
        By our assumption, there is some ordinal \(\gamma\) 
        such that there is an \(\OD_C\)
        bijection \(f : \gamma\to P^{\OD_C}(P(\beta))\).
        Let \(\mathbb B\) be the Boolean algebra on \(\gamma\) induced by the
        Boolean algebra structure on \(P^{\OD_C}(P(\beta))\). 
        Then \(\mathbb B\) is a complete Boolean algebra
        in \(\HOD_C\). Let \(U\subseteq P^{\OD_C}(P(\beta))\)
        be the principal ultrafilter on \(P(\beta)\)
        concentrated at \(x\); that is, 
        \[U = \{X\in P^{\OD_C}(P(\beta)) : x\in X\}\]
        Let \(G = f^{-1}[U]\). 
        Then \(G\subseteq \mathbb B\) is a \(\HOD_C\)-generic ultrafilter.
        Moreover, one can check that \(\HOD_{C\cup \{x\}} = \HOD_C[G]\).
    \end{proof}
\end{thm}

\begin{thm}\label{thm:hcd_ground}
    If \(\kappa\) is strongly compact, then \(\HCD(\kappa)\) is a ground of \(V\).
    \begin{proof}
        The hypotheses of \cref{thm:vopenka} hold
        for \(C = \UF_\kappa(\Ord)\) by \cref{cor:uf_wo}.
        Fix \(x\subseteq \kappa\) such that \(V_\kappa\subseteq L[x]\).
        Then \(\HCD(\kappa)\) is a ground of 
        \(\HCD(\kappa)_x\) by \cref{thm:vopenka}
        and \(V = \HCD(\kappa)_x\) by \cref{prp:v_hcd_x},
        so \(\HCD(\kappa)\) is a ground of \(V\).
    \end{proof}
\end{thm}
The next proposition follows from the proof of \cref{thm:hcd_ground}.
\begin{prp}\label{prp:vop_bound}
    Suppose \(\kappa\) is strongly compact, let \(\delta = (2^{2^\kappa})^+\).
    Then \(\HCD(\kappa)\) is a ground of \(V\) for a
    forcing in \(\HCD(\kappa)\) of cardinality less than \(\delta\).\qed
\end{prp}

This yields a new consequence of the Ground Axiom:
\begin{thm}[Ground Axiom]\label{thm:ga_hcd}
    Assume there is a proper class of strongly compact cardinals.
    Then \(V = \HCD\).\qed
\end{thm} 

One can use also use \cref{thm:hcd_ground} to prove that large cardinals 
are downwards absolute to \(\HCD(\kappa)\).
\begin{thm}[\cite{UA}]\label{thm:cov_to_appx}
    Suppose \(\kappa\) is strongly compact and \(M\) is an inner model with the \(\kappa\)-cover property.
    Then \(M\) has the \(\kappa\)-approximation property if and only if every \(\kappa\)-complete
    ultrafilter on an ordinal is amenable to \(M\).\qed
\end{thm}
An inner model \(M\) is a \textit{weak extender model for the supercompactness of \(\kappa\)}
if for all \(\lambda \geq \kappa\), there is a normal fine ultrafilter \(\mathcal U\)
on \(P_\kappa(\lambda)\) such that \(P_\kappa(\lambda)\cap M\in \mathcal U\)
and \(\mathcal U\cap M \in M\).
\begin{thm}\label{thm:appxcov}
     If \(\kappa\) is strongly compact, 
     then \(\HCD(\kappa)\) has the \(\kappa\)-approximation and cover properties.
     Moreover, if \(\kappa\) is supercompact, then \(\HCD(\kappa)\)
     is a weak extender model for the supercompactness of \(\kappa\).
     \begin{proof}
        It suffices to show the \(\kappa\)-cover property
        by \cref{thm:cov_to_appx}.
        We proceed by showing that for all cardinals \(\lambda \geq \kappa\), 
        there is a \(\kappa\)-complete
        fine ultrafilter on \(P_\kappa(\lambda)\) concentrating on \(\HCD(\kappa)\).
        It follows that \(\HCD(\kappa)\) has the cover property 
        for subsets of \(\lambda\):
        if \(\sigma\in P_\kappa(\lambda)\),
        then since \(\mathcal U\) is fine, 
        \(\{\tau\in P_\kappa(\lambda) : \sigma\subseteq \tau\}\in \mathcal U\)
        and hence intersects the \(\mathcal U\)-large set
        \(P_\kappa(\lambda)\cap \HCD(\kappa)\); in other words,
        for some \(\tau\in \HCD(\kappa)\), \(\sigma\subseteq \tau\).

        It suffices to find such an ultrafilter for all 
        regular \(\lambda\) large enough that \(\HCD(\lambda)\) is
        stationary correct at \(\lambda\). For such a \(\lambda\),
        there is a stationary partition
        \(\langle S_\alpha : \alpha < \lambda\rangle\) of 
        \(S^\lambda_\omega\) that belongs to \(\HCD(\kappa)\).
        Let \(j : V\to M\) be an elementary embedding
        such that \(\crit(j) = \kappa\) and 
        \(\cf^M(\sup j[\lambda]) < j(\kappa)\). 
        Let 
        \(\langle T_\alpha : \alpha < \lambda\rangle =j(\langle S_\alpha : \alpha < \lambda\rangle)\),
        and in \(M\), let
        \(\sigma\) be the set of \(\alpha < \lambda\)
        such that \(T_\alpha\) reflects to \(\sup j[\lambda]\).

        Then \(j[\lambda]\subseteq \sigma\)
        since \(j[S_\alpha]\subseteq j(S_\alpha)\)
        and \(j[S_\alpha]\) is truly stationary in \(\sup j[\lambda]\).
        Moreover, \(|\sigma|^M < j(\kappa)\).
        To see this, fix a closed cofinal set \(C\subseteq \sup j[\lambda]\) 
        of ordertype less than \(j(\kappa)\). 
        Define \(f : \sigma\to C\) by \(f(\alpha) = \min C\cap S_\alpha\).
        Then \(f\) is injective since \(\langle S_\alpha : \alpha < \lambda\rangle\)
        is a partition. Hence \(|\sigma|^M \leq |C|^M < j(\kappa)\).

        Let \(\mathcal U\) be the ultrafilter
        on \(P_\kappa(\lambda)\) derived from \(j\) using \(\sigma\).
        Then since \(j[\lambda]\subseteq \sigma\), 
        \(\mathcal U\) is a \(\kappa\)-complete fine ultrafilter.
        Since \(\sigma\in j(\HCD(\kappa))\), 
        \(P_\kappa(\lambda)\cap \HCD(\kappa) \in\mathcal U\).

        If \(\kappa\) is supercompact, we could have assumed 
        \(j[\lambda]\in M\), in which case, one can prove
        \(\sigma = j[\lambda]\). Then \(\mathcal U\) 
        is a normal fine ultrafilter on \(P_\kappa(\lambda)\).
        Note that \(\mathcal U\) is definable
        from \(j\), so \(\mathcal U\in \CD(\kappa)\)
        and hence \(\mathcal U\cap \HCD(\kappa)\) 
        belongs to \(\HCD(\kappa)\).
        This suffices to conclude that \(\HCD(\kappa)\) is a weak
        extender model for the supercompactness of \(\kappa\).
     \end{proof}
\end{thm}
This has a number of surprising corollaries. 
For example,
if \(E\) is an \(\HCD(\kappa)\)-extender with
\(\crit(j_E) \geq \kappa\) and 
\(j_E(A)\cap [\len(E)]^{< \omega} \in \HCD(\kappa)\)
for all \(A\subseteq [\len(E)]^{<\omega}\),
then by a theorem of Woodin, \(E\)
actually belongs to \(\HCD(\kappa)\), despite the fact that
\(\HCD(\kappa)\) is defined in terms of ultrafilters and not 
extenders. Is there a direct proof of this fact?

We turn now to the structure of \(\HCD\) itself under large cardinal assumptions.
The proof is based on the proof of Usuba's theorem \cite{Usuba},
although the result does not literally follow from his theorem.
\begin{thm}\label{thm:hcd_extendible}
    Suppose \(\kappa\) is an extendible cardinal. Then 
    \(\HCD(\kappa) = \HCD\).
    \begin{proof}
        Suppose \(A\) is a \(\kappa\)-completely
        definable set of ordinals. We will show
        that for any regular cardinal \(\delta \geq \kappa\), \(A\) 
        is \(\delta\)-completely definable.
        For this, fix \(\lambda > \delta\)
        such that \(x\in (\HCD(\kappa))^{V_\lambda}\)
        and let \(j : V_{\lambda+1}\to V_{j(\lambda) + 1}\)
        be an elementary embedding with \(\crit(j) = \kappa\) and
        \(j(\kappa) > \lambda\). Note that \((\HCD({j(\kappa))})^{V_{j(\lambda)}}
        \subseteq \HCD(\delta)\). But \(j(A)\)
        and \(j\restriction \delta\) belong to \((\HCD({j(\kappa)}))^{V_{j(\lambda)}}\),
        the latter by the stationary splitting argument from \cref{thm:appxcov}.
        It follows that \(A = j^{-1}[j(A)]\in \HCD(\delta)\).
    \end{proof}
\end{thm}
\subsection{Embeddings above an extendible cardinal}
We will need the following consequence of 
Kunen's commuting ultrapowers lemma:
\begin{thm}[Kunen]\label{thm:kunen_commute}
    Suppose \(j : V\to M\) is the ultrapower embedding 
    associated to an extender in \(V_\delta\) and \(i : V\to N\) is 
    an ultrapower embedding with \(\crit(i)\geq \delta\). Then
    \(j(i) = i\restriction M\).\qed
\end{thm}
For a proof, see \cref{lma:commuting_ultrapowers} below.
\begin{lma}\label{lma:cd_agree}
    If \(j_0,j_1:V\to M\) are elementary embeddings 
    associated to extenders in \(V_\delta\), then 
    \(j_0\restriction \CD(\delta) = j_1\restriction \CD(\delta)\).
    \begin{proof}
        Suppose \(x\in \CD(\delta)\).
        By \cref{lma:hcd_ultrapower},
        \(x\) is definable from finitely many ordinal parameters \(\vec \alpha\) 
        in the structure \((V,\in,i)\) for some
        ultrapower embedding \(i : V\to N\). 
        By \cref{thm:Woodin}, \(j_0(\vec \alpha) = j_1(\vec \alpha)\),
        and by \cref{thm:kunen_commute},
        \(j_0(i) = i\restriction M = j_1(i)\). 
        Hence \(j_0(x) = j_1(x)\).
    \end{proof}
\end{lma}

\begin{thm}\label{thm:unique_ground}
    Suppose there is a proper class of strongly compact cardinals.
    Then any two elementary embeddings from the universe into the same
    inner model agree on a ground.
    \begin{proof}
        Fix elementary embeddings \(j_0,j_1 : V\to M\). By \cref{thm:ultrapowers}
        and \cref{cor:almost_ultra},
        there are ultrapower embeddings \(i_0,i_1 : V\to N\)
        and an elementary embedding \(k : N\to M\)
        such that \(j_0 = k\circ i_0\) and \(j_1 = k\circ i_1\).
        Let \(\delta\) be a strongly compact cardinal such that
        \(i_0\) and \(i_1\) are the embeddings associated to
        ultrafilters in \(V_\delta\). By \cref{lma:cd_agree}, 
        \(i_0\) and \(i_1\) agree on \(\CD(\delta)\),
        and hence \(j_0\) and \(j_1\) agree on \(\CD(\delta)\).
        Hence \(j_0\) and \(j_1\) agree on \(\HCD(\delta)\), which is 
        a ground by \cref{thm:hcd_ground}.
    \end{proof}
\end{thm}

\begin{thm}\label{thm:ga}
    Assume the Ground Axiom and a proper class of strongly compact cardinals.
    Then the uniqueness of elementary embeddings holds.\qed
\end{thm}

\begin{thm}\label{thm:extendible_uniqueness}
    The uniqueness of elementary embeddings holds above the least extendible cardinal.
    \begin{proof}
        By \cref{thm:ultrapowers}, \cref{cor:almost_ultra},
        and the fact that the eventual singular cardinals hypothesis holds
        (\cref{lma:solovay_SCH}), it suffices to prove the 
        uniqueness of elementary embeddings for ultrapower embeddings 
        \(j_0,j_1 : V\to M\) whose
        critical points lie above the least extendible cardinal \(\kappa\). 
        By \cref{lma:cd_agree}, 
        \(j_0\restriction \HCD = j_1\restriction \HCD\). By \cref{thm:hcd_extendible},
        \(\HCD(\kappa) = \HCD\). Hence \(j_0\) and \(j_1\) agree on \(\HCD(\kappa)\).
        Let \(i\) be their common restriction to \(\HCD(\kappa)\), and note that
        \(i\) is definable over \(\HCD(\kappa)\) by the downwards
        L\'evy-Solovay theorem. Hence \(i\) extends uniquely to any forcing
        extension of \(\HCD(\kappa)\) by a forcing of size less than \(\crit(i)\).
        But by \cref{prp:vop_bound}, \(V\) is such a forcing extension.
        Thus \(j_0 = j_1\), since each extends \(i\).
    \end{proof}
\end{thm}
It would be interesting to know whether there is a combinatorial proof of
this theorem avoiding the use forcing and ordinal definability.

We say the uniqueness of elementary embeddings \textit{fails cofinally}
if it fails above every cardinal.
\begin{thm}\label{thm:superconsistency}
    It is (relatively) consistent with a proper class of supercompact cardinals
    that the uniqueness of elementary embeddings fails cofinally.
    \begin{proof}[Sketch]
    First class force to make every supercompact cardinal indestructible by
    \(\kappa\)-directed closed forcing. 
    Let \(\mathbb P\) the (class) Easton product of the forcings adding a Cohen subset of every
    inaccessible non Mahlo cardinal \(\kappa\). This preserves the supercompacts
    by standard arguments. Moreover, for each \(\kappa\) of Mitchell order 1,
    one can factor \(\mathbb P\) as \(\mathbb P_{\kappa,\infty}\times \mathbb P_{0,\kappa}\)
    and run a essentially the same argument as \cref{thm:consistency}
    in \(V^{\mathbb P_{\kappa,\infty}}\) 
    to prove the uniqueness of elementary embeddings fails at \(\kappa\)
    in \(V^\mathbb P\).
    \end{proof}
\end{thm}
Note that a model with a proper class of strongly compact cardinals
in which the uniqueness of elementary embeddings fails cofinally
must have a proper class of grounds by \cref{thm:unique_ground}.

\subsection{Application: the Kunen inconsistency for ultrapowers}
This section concerns the following open question.
\begin{qst}\label{qst:kunen_ultrapowers}
    Suppose \(j : V\to M\) is an elementary embedding 
    such that \(M\) is closed under \(\omega\)-sequences.
    Can there be a nontrivial elementary embedding \(k :M \to M\)?
\end{qst}
Note that the requirement that \(M\) be closed under \(\omega\)-sequences is necessary
since given any elementary embedding \(j\), one can construct an iterated ultrapower
\(j_{0\omega} : V\to M_\omega\) such that \(j\) restricts to an elementary embedding
from \(M_\omega\) to itself.
\begin{defn}
    The Rudin-Keisler order is defined on extenders \(E\) and \(F\) 
    by setting \(E\sRK F\) if there is a nontrivial elementary embedding \(k : M_E\to M_F\)
    such that \(k\circ j_E = j_F\).
\end{defn}
Combinatorially, \(E\sRK F\) if there is a nonidentity function \(g : \len(E)\to \len(F)\)
such that \(E_{a} = F_{g[a]}\).

The Rudin-Keisler order is a preorder, 
and it is well-known that if \(E\) is the extender of an ultrafilter
then \(E\not\sRK E\). In other words, if \(j: V\to M\) is an ultrapower embedding,
then there is no nontrivial elementary embedding
\(k : M\to M\) such that \(k\circ j = j\). In fact, a theorem of Solovay states
that the Rudin-Keisler order is \textit{wellfounded} on countably complete ultrafilters.
We will generalize this to extenders. The issue, however, is that the Rudin-Keisler order
is \textit{not} wellfounded, or even irreflexive, on arbitrary (wellfounded) extenders.
(See the remarks following \cref{qst:kunen_ultrapowers}.)
\begin{defn}
An extender \(E\) is said to be \textit{countably closed}
if its associated ultrapower \(M_E\) is closed under \(\omega\)-sequences.
\end{defn}
\begin{thm}\label{thm:RK_WF}
    The Rudin-Keisler order is wellfounded on countably closed extenders.
\end{thm}
Actually, in the spirit of this paper, we prove a slightly stronger second-order theorem
(\cref{thm:inv_lim}), although \cref{lma:extender} suggests that this extra strength is an illusion.
\begin{defn}
    Suppose \(P\) and \(Q\) are models of \text{ZF}
    and \(j : P \to Q\) is a cofinal elementary embedding.
    For \(a,b\in Q\), set \(a\leq_j b\) if there is a structure
    \(\mathcal M\in P\) such that \(b\in j(\mathcal M)\) and 
    \(a\) is definable in \(j(\mathcal M)\)
    using \(b\) as a parameter. For any \(a\in Q\), let \(\nu_j(a)\) denote 
    the least ordinal \(\nu\) such that \(a \leq_j \nu\).
\end{defn}
By reflection, one can prove the schema: if \(a\) is definable in \(Q\) from \(b\) and parameters
in \(j[P]\), then \(a\leq_j b\). If \(P\) and \(Q\) are models of ZFC, one can prove that \(\nu_j(a)\)
is defined for all \(a\in Q\) using the Wellordering Theorem. Using Los's Theorem, one can show that
for any \(b\in Q\), the set \(X_b = \{a\in Q : a\leq_j b\}\) is an elementary substructure
of \(Q\).

The following remarkable fact about elementary embeddings of transitive models
of ZFC may be due to Solovay. In any case, it is closely related to his proof of the wellfoundedness
of the Rudin-Keisler order on countably complete ultrafilters.
\begin{lma}[Folklore]\label{lma:j_wf}
    Suppose \(P\) and \(Q\) are wellfounded models of \textnormal{ZFC}
    and \(j : P \to Q\) is an elementary embedding.
    Then \(\leq_j\) is wellfounded and has rank bounded by \(\Ord\cap Q\).
    \begin{proof}
        Let \(\nu = \nu_j(a)\). We first show that \(\nu \leq_j a\). 
        Let \(D\) be the \(P\)-ultrafilter derived from \(j\) using \(a\), 
        let \(k : M_D^P\to Q\) be the canonical factor embedding,
        and let \(\bar a = k^{-1}(a)\).
        Let \(\bar \nu = \nu_{j_D}(\bar a)\). Thus \(\bar a \leq_{j_D}\bar \nu\) 
        and moreover \(\bar \nu\leq_{j_D}(\bar a)\)
        since this is true of every \(x\in M_D^P\). 
        It follows that \(k(\bar \nu) \leq_j a\) and \(a \leq_j k(\bar \nu)\).
        The latter fact implies \(k(\bar \nu)\geq \nu\).
        Assume towards a contradiction that \(k(\bar \nu) > \nu\).
        Fix a structure \(\mathcal M\in P\) such that \(\nu \in j(\mathcal M)\) and \(a\) is 
        definable from \(\nu\) in \(j(\mathcal M)\). 
        Then \(Q\) satisfies that there is an ordinal \(\xi < k(\bar \nu)\)
        such that \(a\) is definable from \(\xi\) in \(j(\mathcal M)\).
        By elementarity, 
        \(M_D^{P}\) satisfies that there is an ordinal \(\xi < \bar \nu\)
        such that \(\bar a\) is definable from \(\xi\) in \(j_D(\mathcal M)\).
        This contradicts that \(\bar \nu = \nu_{j_D}(\bar a)\).

        It follows that the function \(\nu_j : Q\to \Ord^Q\) ranks the preorder \(\leq_j\).
        Indeed, suppose \(a <_j b\) (in the sense that \(a\leq_j b\) but \(b\not\leq_j a\)).
        Obviously \(\nu_j(a) \leq \nu_j(b)\) (since \(\leq_j\) is transitive).
        But \(\nu_j(a)\neq \nu_j(b)\) or else \(b \leq_j \nu_j(b) = \nu_j(a) \leq_j a\).
        Thus \(\nu_j(a) < \nu_j(b)\).
    \end{proof}
\end{lma}
\begin{thm}\label{thm:inv_lim}
    Suppose \(\langle j_n : n < \omega\rangle\) is a sequence of elementary 
    embeddings \(j_n : V\to M_n\) where \(M_n\) is closed under \(\omega\)-sequences.
    Suppose \(\langle i_{n,m} : m \leq n < \omega\rangle\) is a sequence
    of elementary embeddings such that
    for all \(\ell \leq m \leq n < \omega\), 
    \begin{align*}
        i_{m,\ell}\circ i_{n,m} &= i_{n,\ell}\\
        i_{n,m}\circ j_n &= j_m
    \end{align*}
    Then for all sufficiently large \(m\leq n < \omega\), \(M_m = M_n\) and \(i_{n,m}\)
    is the identity.
    \begin{proof}
        For each \(n < \omega\), let \(\kappa_n = \crit(i_{n+1,n})\).
        Consider the sequence \(s = \langle i_{n,0}(\kappa_n) :  n < \omega\rangle\).
        For each \(n < \omega\), let \(s_n = s\restriction [n,\omega)\).
        Since \(M_0\) is closed under \(\omega\)-sequences, \(s_n\in M_0\) for all \(n < \omega\).

        Clearly \(s_{n+1}\leq_j s_{n}\). We claim that \(s_n\not \leq_j s_{n+1}\).
        Note that \(s_{n+1}\in i_{n+1,0}[M_{n+1}]\).
        Since \(j_0[V]\subseteq i_{n+1,0}[M_{n+1}]\) and \(i_{n+1,0}[M_{n+1}]\) is definably closed,
        \(i_{n+1,0}[M_{n+1}]\) is downwards closed under \(\leq_j\).
        Now \(i_{n,0}(\kappa_n)\notin i_{n+1,0}[M_{n+1}]\) 
        since 
        \[i_{n,0}^{-1}(i_{n,0}(\kappa_n)) = 
        \kappa_n\notin i_{n+1,n}[M_{n+1}] = i_{n,0}^{-1}[i_{n+1,0}[M_{n+1}]]\]
        But \(\kappa_n = s_n(n) \leq_j s_n\). Hence \(s_n\notin i_{n+1,n}[M_{n+1}]\).
        In particular, \(s_n\not \leq_j s_{n+1}\).

        Thus for all \(n < \omega\), \(s_{n+1} <_{j_0} s_n\), 
        and this contradicts \cref{lma:j_wf}.
    \end{proof}
\end{thm}
This yields \cref{thm:RK_WF}:
\begin{proof}[Proof of \cref{thm:RK_WF}]
    Take ultrapowers and apply \cref{thm:inv_lim}.
\end{proof}
\begin{cor}
    Suppose \(M\) is an inner model closed under \(\omega\)-sequences 
    and \(j : V\to M\) and \(k : M\to M\) are elementary embeddings
    such that \(k\circ j = j\). Then \(k\) is the identity.
    \begin{proof}
        Apply \cref{thm:inv_lim}.
    \end{proof}
\end{cor}
\begin{thm}
    Suppose \(\delta\) is extendible, \(M\) is an inner model closed under \(\omega\)-sequences, 
    and \(j : V\to M\) is an elementary embedding with critical point above \(\delta\). 
    Then there is no nontrivial elementary embedding from \(M\) to \(M\).
    \begin{proof}
        Suppose \(k : M\to M\) is an elementary embedding. Note that \(k\circ j\) and \(j\)
        agree on the ordinals by \cref{thm:unique_ords}, and therefore \(\crit(k\circ j) > \delta\).
        The theorem now follows from \cref{thm:extendible_uniqueness}.
    \end{proof}
\end{thm}
\subsection{Weaker hypotheses}
The models \(\HCD(\kappa)\) of the previous section are particularly interesting,
being models of the Axiom of Choice, but in fact certain applications
of these models can be carried out under hypotheses of lower consistency strength
than a strongly compact cardinal.
Here we will show:
\begin{thm}\label{thm:strong}
    Assume there is a proper class of strong cardinals. Then
    any pair of ultrapower embeddings from the universe into an inner model
    agree on a ground of \(V\).
\end{thm}
We immediately obtain a proof of the uniqueness of elementary embeddings
from the Ground Axiom under (consistency-wise) weaker hypotheses:
\begin{cor}\label{thm:strong_ga}
    Assume the Ground Axiom and a proper class of strong cardinals.
    Then the uniqueness of ultrapower embeddings holds.\qed
\end{cor}
By \cref{thm:ultrapowers}, 
the uniqueness of ultrapower embeddings implies
the uniqueness of extender embeddings.
In the context of the eventual SCH, one can improve this to arbitrary elementary embeddings
using \cref{cor:almost_ultra}.

We begin by defining a ZF-ground of \(V\) on which the embeddings agree.
\begin{defn}
    Suppose \(\kappa\) is a cardinal. A set is \textit{\(\kappa\)-extender definable}
    if it is definable over \((V,j)\) for some 
    extender embedding \(j : V\to M\) such that \(\crit(j)\geq \kappa\)
    and \(M^{<\kappa} \subseteq M\). We denote the class of \(\kappa\)-extender definable
    sets by \(\ED(\kappa)\). The class of \textit{hereditarily \(\kappa\)-extender definable sets},
    denoted by \(\HED(\kappa)\), is the largest transitive subclass of \(\ED(\kappa)\).
\end{defn}
Everything we prove about \(\ED(\kappa)\) can also be proven
about the conceivably smaller class of sets definable from short strong extenders,
or in other words definable from an elementary embedding 
\(j : V\to M\) such that \(\crit(j)= \kappa\), \(j(\kappa) > \lambda\)
and \(M^\kappa \subseteq M\), \(V_\lambda\subseteq M\), and \(M = H^M(j[V]\cup V_\lambda)\).
The relationship between the two notions is unclear.

The proof uses a generalization of Kunen's commuting ultrapowers lemma:
\begin{lma}\label{lma:commuting_ultrapowers}
    Suppose \(i : V\to M\) and \(j : V\to N\) are elementary embeddings such that
    \(i(j) = j\restriction M\) and \(i(\nu) = j(i)(\nu)\) 
    for all generators \(\nu\) of \(j\). 
    Then \(j(i) = i\restriction N\).
    \begin{proof}
        In general given elementary embeddings \(i : V\to M\) and \(j : V\to N\),
        one has \(i\circ j = i(j)\circ i\) because \(i(j(x)) = i(j)(i(x))\).
        But note that in our case, 
        \(i\circ j = j(i)\circ j\)
        since \(i\circ j = i(j)\circ j = j\circ i = j(i)\circ j\).
        In particular, \(i(N) = j(i)(N)\) since 
        \(N = j(V)\).
        Therefore \(i,j(i) : N\to i(N)\) are elementary embeddings with the same target model.

        Note that \(i\restriction j[V] = j(i)\restriction j[V]\), since this is just another
        way of saying that \(i\circ j = j(i)\circ j\).
        Our hypothesis states that \(i\restriction G = j(i)\restriction G\)
        where \(G\) is the class of generators of \(j\).
        Now \(N = H^N(j[V]\cup G)\),
        and \(i\) and \(j(i)\) coincide on \(j[V]\cup G\). Hence \(i = j(i)\).
    \end{proof}
\end{lma}
\begin{cor}\label{cor:ed_agree}
    Suppose \(\lambda\) is a cardinal and 
    \(j_0,j_1 : V\to M\) are extender embeddings such that \(M= H^M(j_n[V]\cup V_\alpha)\),
    for some \(\alpha < \lambda\).
    Then \(j_0\) and \(j_1\) agree on \(\ED(\lambda)\).\qed
\end{cor}
The following theorem, generalizing \cref{thm:hcd_ground}, is the key:
\begin{thm}\label{thm:zf_ground}
    Suppose \(\kappa\) is strong. Then 
    \(\HED(\kappa)\) is a \ZF-ground.
\end{thm}
The theorem cannot be proved in exactly the same way as \cref{thm:hcd_ground} 
since Vop\v{e}nka's theorem does not seem to go through. But one can instead use Bukovsky's Theorem. 
Suppose \(\theta\) is a cardinal.
An inner model \(M\) is said to have the \textit{\(\theta\)-uniform cover property}
for all \(X\in M\) and \(f : X\to M\), there
is a function \(F : X\to M\) in \(M\) such that
\(f(x)\in F(x)\) for all \(x\in X\) and 
\(F(x)\) does not surject onto \(\theta\).

The following proposition is a version of Bukovsky's theorem
which follows the proof from \cite{FriedmanSakai} in order to deal with ZF-grounds.
Our situation is nominally different, since our definition of the \(\theta\)-uniform cover property
is somewhat weaker than the one employed there.
\begin{prp}\label{prp:bukovsky}
    Suppose \(M\) is an inner model of \ZF\ with the 
    \(\theta\)-uniform cover property for some cardinal \(\theta\).
    Then every set of ordinals is generic over \(M\).
    \begin{proof}
        Let \(\gamma\) be an ordinal and \(A\) a subset of \(\gamma\). 
        We will show \(A\) is generic over \(M\).

        Let \(\mathcal L\)
        denote the class of infinitary propositional formulae in with \(\gamma\)
        indeterminates \(\langle x_\alpha : \alpha < \gamma\rangle\).
        Let \(\mathcal L_M = \mathcal L\cap M\).
        Let \(\lambda > \gamma\) be a Beth fixed point of cofinality
        at least \(\theta\).
        Let \(\mathcal L_\lambda = \mathcal L_M \cap V_\lambda\).
        Let \(f : P(\mathcal L_\lambda)\to\mathcal L_\lambda\)
        assign to each \(\Gamma\subseteq \mathcal L_\lambda\)
        such that \(A\vDash \bigvee \Gamma\)
        some \(\varphi\in \Gamma\) such that \(A\vDash \varphi\).
        Let \(F : P(\mathcal L_\lambda)\to P(\mathcal L_\lambda)\) be a function in \(M\)
        witnessing the \(\theta\)-uniform cover property for \(f\).
        Note that our assumption that \(\cf(\lambda) \geq \theta\)
        yields that \(\bigvee F(\Gamma) \in \mathcal L_\lambda\)
        for all \(\Gamma\subseteq \mathcal L_\lambda\).

        Let \(T\) be the theory consisting of formulae of the form 
        \(\bigvee \Gamma\to \bigvee F(\Gamma)\) for 
        nonempty \(\Gamma\subseteq \mathcal L_\lambda\).
        Note that \(A\vDash T\).
        Let \(\mathbb P\) be the set of \(\varphi\in \mathcal L_\lambda\) 
        such that \(T\) does not prove \(\neg\varphi\)
        (using any valid proof system for infinitary logic 
        that is definable in \(M\)
        and suffices for the argument in the final paragraph of the
        proof of this proposition). Partially order
        \(\mathbb P\) by setting
        \(\varphi \leq \psi\) if \(T\vdash \varphi\to \psi\).

        Let \(G\subseteq \mathbb P\) be the set of
        \(\varphi\in \mathbb P\) such that \(A\vDash \varphi\).
        We claim \(G\) is an \(M\)-generic filter. We leave the verification that
        \(G\) is a filter to the reader. (See \cite{FriedmanSakai}.) 
        
        Suppose \(D\subseteq \mathbb P\) 
        is a dense set that lies in \(M\). We claim \(A\vDash \bigvee D\).
        Otherwise \(T\) does not prove \(\bigvee D\). Since \(F(D)\subseteq D\),
        it follows that \(T\) does not prove
        \(\bigvee F(D)\). Let \(\varphi = \bigvee F(D)\). 
        Then \(\neg\varphi\in \mathbb P\).
        By density, fix \(\psi\in D\) such that \(\psi\leq \neg\varphi\). 
        By contraposition, \(T\) proves \(\neg\varphi\) implies \(\neg\left(\bigvee D\right)\).
        Therefore since \(T\vdash \psi\to \neg\varphi\),
        \(T\vdash \psi\to \neg(\bigvee D)\). Since
        \(\psi\in D\), \(T\vdash \neg(\bigvee D)\to \neg\psi\),
        and therefore \(T\vdash \psi\to \neg\psi\).
        As a consequence, \(T\vdash \neg\psi\), contradicting that \(\psi\in \mathbb P\).
    \end{proof}
\end{prp}

Relativizing extender definability gives rise to the classes \(\ED_A\)
and \(\HED_A\) for every set parameter \(A\).
\begin{prp}\label{prp:hed_A}
    Suppose \(\kappa\) is strong and \(A\subseteq \kappa\) is such that
    \(V_\kappa\subseteq\HED(\kappa)_A\).
    Then \(V = \HED(\kappa)_A\).
    \begin{proof}
        Fix \(\lambda \geq \kappa\), and we will show \(V_\lambda\subseteq \HED(\kappa)_A\).
        For this, let \(j: V\to M\) be an elementary embedding such that 
        \(\crit(j) = \kappa\), \(j(\kappa) > \lambda\), 
        \(V_\lambda\subseteq M\), and \(M^{<\kappa}\subseteq M\). 
        Then \(j(A)\in \HED(\kappa)_A\)
        and \(V_\lambda \subseteq \HED^M(j(\kappa))_{j(A)} \subseteq \HED(\kappa)_{j(A)}\).
    \end{proof}
\end{prp}
\begin{proof}[Proof of \cref{thm:zf_ground}]
    By \cref{prp:bukovsky} and \cref{prp:hed_A}, it suffices to
    prove that \(\HED(\kappa)\) has the \(\theta\)-uniform cover property
    inside \(\HED(\kappa)_A\) for some \(\theta\). Let \(\theta = (2^\kappa)^+\). 
    
    Suppose \(f : X\to \HED(\kappa)\)
    is \(\ED(\kappa)_A\), and we will find \(F : X\to \HED(\kappa)\) 
    in \(\HED(\kappa)\) witnessing the \(\theta\)-uniform cover property.
    Fix an extender \(E\) and a formula \(\varphi\)
    such that \(f(x) = y\) if and only if \(\varphi(x,y,A,E)\) holds.
    Define \(F : X\to \HED(\kappa)\) by setting 
    \[F(x) = \{y : \exists B\subseteq \kappa\, \varphi(x,y,B,E)\}\]
    Clearly \(F\) is \(\ED(\kappa)\), and so \(F\in \HED(\kappa)\).
    In \(V\), \(F(x)\) is the surjective image of \(P(\kappa)\),
    and so in \(\HED(\kappa)\), \(F(x)\) does not surject onto \((2^\kappa)^+\).
    Since \(f(x) \in F(x)\) for all \(x\in X\), \(F\) is as desired.
\end{proof}

\begin{cor}\label{cor:zfc_ground}
    If \(\kappa\) is a strong cardinal, then \(\HED(\kappa)\) contains a ground.
    \begin{proof}
        A theorem of Usuba \cite{UsubaZF} shows (in ZFC) that every ZF-ground contains
        a ground.
    \end{proof}
\end{cor}

\begin{proof}[Proof of \cref{thm:strong}]
    The theorem is now immediate from \cref{cor:ed_agree} and \cref{cor:zfc_ground}.
\end{proof}
\section{Ultrapower axioms}\label{UASection}
The Ultrapower Axiom is a combinatorial principle that clarifies the theory of
countably complete ultrafilters. Here we will show it implies the uniqueness of
elementary embeddings. We will also consider a slight weakening of the Ultrapower Axiom
called the Weak Ultrapower Axiom, which until this work 
was not known to have any consequences at all.

An elementary embedding \(i : P\to Q\) is \textit{close} if 
for all \(A\in Q\), \(i^{-1}[A]\in P\). If \(i\) is an ultrapower
embedding, we say in this case that \(i\) is \textit{internal},
since for ultrapower embeddings, closeness is equivalent to
the existence of an ultrafilter \(U\in P\) and 
an isomorphism \(k : \Ult(P,U)\to Q\) such that \(k\circ j_U= i\).
\begin{defn}
    The Ultrapower Axiom states that for any inner models \(P_0\) and \(P_1\) 
    admitting internal ultrapower embeddings \(j_0 : V\to P_0\)
    and \(j_1 : V\to P_1\), there exists an inner model \(N\) admitting 
    internal ultrapower embeddings \(k_0 : P_0\to N\)
    and \(k_1 : P_1\to N\) such that \(k_0\circ j_0 = k_1\circ j_1\).
\end{defn}
Although it is not immediate from our formulation here,
UA is a first-order statement. In fact, it is equivalent to a \(\Pi_2\)-sentence.
\subsection{The Ultrapower Axiom}
In this subsection, we show that UA implies the uniqueness of elementary embeddings.
\begin{thm}\label{thm:ua}
    Assume the Ultrapower Axiom. Then the uniqueness of elementary embeddings holds.
\end{thm}
The Ultrapower Axiom in general is only really useful for ultrapower embeddings,
so the generality of this theorem may seem surprising until one realizes 
that by \cref{thm:ultrapowers},
at least under cardinal arithmetic assumptions,
the uniqueness of elementary embeddings reduces 
to the uniqueness of ultrapower embeddings,
which is easily proved from UA:
\begin{lma}[UA]\label{lma:ua_unique_ultrapowers}
    Suppose \(j_0,j_1 : V\to M\) are ultrapower embeddings.
    Then \(j_0 = j_1\).
    \begin{proof}
        Let \((k_0,k_1) : M\to N\) be an internal ultrapower comparison of 
        \((j_0,j_1)\). Note that 
        \(k_0\circ j_0 = k_1\circ j_1\) by the definition of a comparison
        and \(k_0\) and  \(k_1\) agree on the ordinals
        by \cref{thm:Woodin}. Since every set is constructible from a set
        of ordinals, it suffices to show that for all sets of ordinals
        \(A\), \(j_0(A) = j_1(A)\). But 
        \[j_0(A) = k_0^{-1}[k_0(j_0(A))] = k_0^{-1}[k_1(j_1(A))] 
        = k_1^{-1}[k_1(j_1(A))] = j_1(A)\qedhere\]
    \end{proof}
\end{lma}
All that remains to prove the uniqueness of elementary embeddings from UA
is to prove that the conclusion of \cref{lma:extender} 
follows from UA without appealing to SCH.
\begin{lma}[UA]\label{lma:ua_extender}
    Suppose \(M\) is a countably closed inner model and \(j : V\to M\) is an elementary embedding. 
    Then \(j\) is an extender embedding.
    \begin{proof}
        Suppose not. Then there is a strong limit cardinal \(\lambda\) of cofinality \(\omega\)
        that is closed under \(j\) and a limit of generators of \(j\). 
        Since \(j\) is continuous at ordinals of cofinality \(\omega\), \(j(\lambda) = \lambda\).
        Let \(\langle \nu_n : n < \omega\rangle\) be an increasing sequence
        of generators of \(j\) whose limit is \(\lambda\).
        Let \(U\) be the ultrafilter on \(\lambda^\omega\) derived from \(j\)
        using \(\langle \nu_n : n < \omega\rangle\). Then \(\lambda < \lambda_U \leq \lambda^\omega\).
        Let \(\gamma = \lambda^\sigma\) be the least cardinal greater than \(\lambda\)
        that carries a countably complete uniform ultrafilter. By \cite{UA},
        since \(\lambda\) is a strong limit cardinal, \(\gamma\) is either measurable
        or \(\gamma = \lambda^+\). Since \(\gamma \leq \lambda_U \leq \lambda^\omega\),
        \(\gamma = \lambda^+\). Since \(\lambda^+\) carries a countably complete uniform ultrafilter,
        \cref{lma:solovay_SCH} implies \(2^\lambda = \lambda^+\). 
        It follows that \(\lambda_U = \lambda^+\),
        but \(j_U(\lambda^+) \leq j(\lambda^+) = \lambda^+\), which is a contradiction.
    \end{proof}
\end{lma}

\begin{proof}[Proof of \cref{thm:ua}]
    By \cref{thm:ultrapowers},
    one can reduce to proving the uniqueness of embeddings that are almost ultrapowers,
    but by \cref{lma:ua_extender}, if an embedding is almost an ultrapower, it actually is an ultrapower.
    By \cref{lma:ua_unique_ultrapowers}, 
    the uniqueness of ultrapower embeddings is a consequence of UA.
\end{proof}
\subsection{The Weak Ultrapower Axiom}
A model \(Q\) is an \textit{internal ultrapower} of a model \(P\) if there is an internal ultrapower
embedding from \(P\) to \(Q\).
In slogan form, the Ultrapower Axiom states: \textit{any two ultrapowers of the universe
have a common internal ultrapower.} Like so many slogans, this is not completely accurate,
since the Ultrapower Axiom contains an additional requirement
amounting to the commutativity of a certain diagram of ultrapowers. 
This discrepancy raises a number of questions.
\begin{defn}
    The \textit{Weak Ultrapower Axiom} states that any two ultrapowers of the
    universe have a common internal ultrapower.
\end{defn}
By ultrapower, we here mean \textit{wellfounded} ultrapower. While UA is \(\Pi_2\),
it is not clear whether Weak UA is. It is first-order, though, and in fact it is
\(\Pi_3\).

Does the Weak Ultrapower Axiom imply the Ultrapower Axiom? Assuming the uniqueness of 
elementary embeddings, the answer is obviously yes.
\begin{prp}
    The Ultrapower Axiom is equivalent to the conjunction of the Weak Ultrapower Axiom
    and the uniqueness of ultrapower embeddings.\qed
\end{prp}
Using the results of this paper, one can prove some of the consequences of UA
assuming just Weak UA by increasing the large cardinal hypotheses.
We only sketch the proofs.
\begin{defn}
    If \(M_0\) and \(M_1\) are inner models and \(\alpha_0\) and \(\alpha_1\)
    are ordinals, we write \((M_0,\alpha_0) \sim (M_1,\alpha_1)\) if there 
    exist elementary embeddings \(k_0 : M_0\to N\) and \(k_1 : M_1\to N\)
    to a common inner model \(N\)
    such that \(k_0(\alpha_0) = k_1(\alpha_1)\).
\end{defn}
It is unclear whether this relation is first-order definable, but this will
not be an issue.
\begin{thm}\label{thm:sim}
    Suppose \(\kappa\) is an extendible cardinal and \(U_0\) and \(U_1\) are \(\kappa^+\)-complete 
    ultrafilters on ordinals 
    such that \((M_{U_0},\id_{U_0}) \sim (M_{U_1},\id_{U_1})\). Then \(U_0 = U_1\).
    \begin{proof}
        For \(n = 0,1\), let \(j_n :V\to M_n\)
        be the ultrapower embedding associated to \(U_n\) and let
        \(\alpha_n = \id_{U_n}\). 
        Let \(N\) be a model of set theory 
        admitting elementary embeddings \(k_0 : M_0\to N\) and \(k_1 : M_1\to N\)
        such that \(k_0(\alpha_0) = k_1(\alpha_1)\).
        Note that
        \(k_0\circ j_0\) and \(k_1\circ j_1\) agree on \(\HCD\)
        by \cref{lma:cd_agree}.
        Therefore \(W = U_0\cap \HCD = U_1\cap \HCD\). Since \(W\) is \(\kappa^+\)-complete and 
        \(\HCD = \HCD(\kappa)\),
        \(W \in \HCD\). Since \(V\) is a generic extension of \(\HCD\)
        for a forcing of size less than the completeness of \(W\) (\cref{prp:vop_bound}), 
        the upwards L\'evy-Solovay theorem implies that
        \(U_0\) is the filter generated by \(W\). Similarly,
        \(U_1\) is the filter generated by \(W\), so \(U_0 = U_1\).
    \end{proof}
\end{thm}
\cref{thm:sim} enables us to define a wellorder of the \(\kappa^+\)-complete ultrafilters.
\begin{defn}
    If \(M_0\) and \(M_1\) are inner models and \(\alpha_0\) and \(\alpha_1\)
    are ordinals, then  \[(M_0,\alpha_0) \ke (M_1,\alpha_1)\] if there is an inner model 
    \(N\) admitting an elementary embedding \(k_0 : M_0\to N\) and
    an internal ultrapower embedding \(k_1 : M_1\to N\)
    with \(k_0(\alpha_0) < k_1(\alpha_1)\).
    The \textit{weak Ketonen order} is defined on countably complete ultrafilters \(U_0\) and \(U_1\) by
    setting \(U_0 \ke^* U_1\) if \((M_{U_0},\id_{U_0})\ke (M_{U_1},\id_{U_1})\).
\end{defn}
Since we did not require that \(k_0\) is an ultrapower embedding,
it is unclear whether the weak Ketonen order is first-order definable, but under the large cardinal hypotheses
we are assuming (or simply the eventual SCH), \cref{thm:ultrapowers} implies
that \(k_0\) must be an ultrapower embedding. (Actually, under Weak UA with no cardinal 
arithmetic hypothesis, one can show that the weak Ketonen order 
is always witnessed by a pair of internal ultrapower embeddings.)
In either context, it follows that
the Ketonen order on models is wellfounded. By \cref{thm:sim},
this yields:
\begin{cor}[Weak UA] If \(\kappa\) is an extendible cardinal,
    the class of \(\kappa^+\)-complete ultrafilters on ordinals is wellordered by the
    weak Ketonen order.\qed
\end{cor}

\begin{cor}[Weak UA]\label{thm:wua_od}
    If \(\kappa\) is an extendible cardinal,
    every \(\kappa^+\)-complete ultrafilter is ordinal definable.\qed
\end{cor}

\begin{thm}[Weak UA]
    If there is an extendible cardinal, then \(V\) is a generic extension of \(\HOD\).
    \begin{proof}
        Since there is an extendible cardinal,
        \(V\) is a generic extension of \(\HCD\) by \cref{thm:hcd_ground} and 
        \cref{thm:hcd_extendible}.
        By \cref{thm:wua_od}, \(\HCD = \HOD\).
    \end{proof}   
\end{thm}

We now bound the size of the forcing taking \(\HOD\) to \(V\).
Somewhat surprisingly, one can show that it is \textit{strictly smaller}
than the least extendible.
\begin{thm}[Weak UA]\label{thm:weak_ua_bound}
    Suppose there is an extendible cardinal. Then \(V\) is a generic extension of \(\HOD\)
    by a forcing in \(V_\delta\) where \(\delta\) is the least \(\Sigma_3\)-reflecting cardinal.
\end{thm}

Note that if \(U_0 \sim U_1\), then \(U_0\cap \HOD = U_1\cap \HOD\).
Combining this with the fact that the weak Ketonen order is a prewellorder
whose induced equivalence relation extends \(\sim\), one obtains:
\begin{lma}[Weak UA]\label{lma:hod_amenable}
    Suppose \(U\) is a countably complete ultrafilter on an ordinal. Then \(U\cap \HOD\in \HOD\).\qed
\end{lma}
We will use this to show that \(\HOD\) has the \(\kappa\)-approximation and cover properties
at the least strongly compact cardinal. This in turn yields that \(\HOD\) is locally
definable from parameters.
\begin{thm}[Weak UA]\label{thm:weak_ua_appx}
    Assume there is an extendible cardinal and let \(\kappa\) be a strongly compact
    cardinal. Then \(\HOD\) has the \(\kappa\)-approximation and cover properties.
    If \(\kappa\) is supercompact, then \(\HOD\) is a weak extender model for the supercompactness
    of \(\kappa\).
    \begin{proof}
        By the strongly compact version of the \(\HOD\) dichotomy theorem \cite{HODNote}, 
        since \(\HOD\) computes sufficiently large successor cardinals, \(\HOD\)
        has the \(\kappa\)-cover property. By \cref{lma:hod_amenable},
        every countably complete ultrafilter on an ordinal is amenable to \(\HOD\).
        Therefore by \cref{thm:cov_to_appx}, 
        \(\HOD\) has the \(\kappa\)-approximation and 
        cover properties. The second part of the theorem is similar to \cref{thm:appxcov}.
    \end{proof}
\end{thm}
Given \cref{thm:weak_ua_appx}, one obtains \cref{thm:weak_ua_bound} 
simply by counting quantifiers.
\begin{proof}[Proof of \cref{thm:weak_ua_bound}]
        Let \(\kappa\) be the least strongly compact cardinal, so \(\kappa < \delta\). 
        Let \(H = \HOD\cap H(\kappa^+)\).
        By Hamkins's pseudoground model definability theorem \cite{Fuchs}, \(\HOD\) is 
        uniformly definable from \(H\) in 
        \(H(\gamma)\) for any strong limit cardinal \(\gamma > \kappa\).
        Therefore the statement that \(V\) is a generic extension  of \(\HOD\) is \(\Sigma_3\)
        in the parameter \(H\), and so it reflects to \(V_\delta\).
        Then taking a generic \(G\in V_\delta\) such that \(V_\delta = (\HOD\cap V_\delta)[G]\),
        the correctness of \(V_\delta\) implies that in fact, \(V = \HOD[G]\). 
\end{proof}
Now repeating the proofs, we can state slightly nicer theorems:
\begin{thm}[Weak UA] If \(\kappa\) is an extendible cardinal,
    the class of \(\kappa\)-complete ultrafilters on ordinals is wellordered by the
    weak Ketonen order. In particular, every \(\kappa\)-complete ultrafilter is ordinal definable.\qed
\end{thm}
\begin{defn}
    We say the Ultrapower Axiom holds for a pair of ultrapower embeddings
    \(j_0 :V\to M_0\) and \(j_1 : V\to M_1\) if there is an inner model
    \(N\) admitting internal ultrapower embeddings \(k_0 : M_0\to N\)
    and \(k_1 : M_1\to N\) such that \(k_0\circ j_0 = k_1\circ j_1\).
\end{defn}
\begin{thm}[Weak UA]\label{thm:hod_ua}
    Suppose \(\kappa\) is an extendible cardinal.
    Then in \(\HOD\), the Ultrapower Axiom holds for any pair of ultrapower embeddings
    with critical point at least \(\kappa\).
    \begin{proof}
        Fix ultrapower embeddings \(j_0 : \HOD\to M_0\) and \(j_1 : \HOD\to M_1\)
        with critical point at least \(\kappa\).
        Since \(V\) is a forcing extension of \(\HOD\)
        for a forcing in \(V_\kappa\), these ultrapower embeddings lift
        to \(j_0^* : V\to M_0^*\) and \(j_1^* : V\to M_1^*\).
        Applying \cref{lma:hod_amenable}, every countably complete ultrafilter is amenable to \(\HOD\),
        so by elementarity, every countably complete
        ultrafilter of \(M_0^*\) (resp.\ \(M_1^*\)) is amenable to \(M_0\)
        (resp.\ \(M_1\)). In particular, any internal ultrapower
        embedding of \(M_0^*\) (resp.\ \(M_1^*\)) 
        restricts to a close embedding of \(M_0\) (resp.\ \(M_1\)).

        Applying the Weak Ultrapower Axiom, fix an inner model \(N^*\) and 
        elementary embeddings \(k_0^* : M_0^*\to N^*\)
        and \(k_1^* : M_1^*\to N^*\). Letting 
        \(k_0 = k_0^*\restriction M_0\) and \(k_1^* = k_1\restriction M_1\),
        the amenability of countably complete ultrafilters to \(\HOD\)
        implies \(k_0\) and \(k_1\) are close to \(M_0\) and \(M_1\).
        Also \cref{thm:Woodin} implies
        \(k_0\circ j_0 = k_1\circ j_1\). Let 
        \(X = H^N(k_0[M_0]\cup k_1[M_1])\), let \(H\) be the transitive collapse
        of \(X\), let \(h : H \to N\) be the inverse of the transitive collapse embedding,
        and let \(i_0 : M_0\to H\) and \(i_1 : M_1\to H\)
        be given by \(i_0 = h^{-1}\circ k_0\) and \(i_1 = h^{-1}\circ k_1\).
        It is then easy to show that \(i_0\) and \(i_1\) are internal ultrapower embeddings
        of \(M_0\) and \(M_1\).
    \end{proof}
\end{thm}

\begin{prp}\label{prp:gch}
    Suppose \(\kappa\) is supercompact and the Ultrapower Axiom holds for 
    embeddings with critical point at least \(\kappa\).
    Then for all cardinals \(\lambda \geq \kappa\), \(2^\lambda = \lambda^+\).
    \begin{proof}
        This follows from the proof of the main theorem of \cite{GCH}.
    \end{proof}
\end{prp}
\begin{thm}[Weak UA]\label{thm:gch}
    If \(\kappa\) is extendible,
    then for all cardinals \(\lambda \geq \kappa\), \(2^\lambda = \lambda^+\).
    \begin{proof}
        By \cref{thm:hod_ua} and \cref{prp:gch}, in \(\HOD\), the Generalized Continuum Hypothesis
        holds at all cardinals greater than or equal 
        to the least extendible cardinal. By \cref{thm:weak_ua_bound},
        \(V\) is a generic extension of \(\HOD\)
        for a forcing of size less than the least extendible cardinal,
        and so the Generalized Continuum Hypothesis holds in \(V\)
        at all cardinals greater than or equal to the least extendible cardinal.
    \end{proof}
\end{thm}
A uniform ultrafilter \(U\) on a cardinal \(\lambda\) is \textit{Dodd sound} if the function
\(E : P(\lambda)\to M_U\) defined by \(E(A) = j_U(A)\cap \id_U\) belongs to \(M_U\).
At least in the context of GCH, 
one can think of Dodd soundness as a generalization of supercompactness:
if \(2^{<\lambda} = \lambda\) and \(\mathcal U\) is a normal fine ultrafilter on 
\(P(\lambda)\),
then there is a unique Dodd sound ultrafilter Rudin-Keisler equivalent to \(\mathcal U\).
(Not every Dodd sound ultrafilter is equivalent to a normal fine ultrafilter.)
\begin{prp}\label{prp:mo}
    Suppose \(\kappa\) is a cardinal such that the Ultrapower Axiom holds for 
    embeddings with critical point at least \(\kappa\).
    Then the Mitchell order is linear on \(\kappa\)-complete
    Dodd sound ultrafilters.\qed
\end{prp}

\begin{thm}[Weak UA]\label{thm:mo}
    If \(\kappa\) is extendible,
    then the Mitchell order is linear on \(\kappa\)-complete
    Dodd sound ultrafilters. In particular, the Mitchell order is linear
    on the class of \(\kappa\)-complete normal fine ultrafilters
    with underlying set \(P_\textnormal{bd}(\lambda)\) for some \(\lambda \geq \kappa\).
    \begin{proof}
        \cref{thm:hod_ua} and \cref{prp:mo} yield the linearity of the Mitchell order
        on \(\kappa\)-complete Dodd sound ultrafilters in HOD. By \cref{thm:weak_ua_bound},
        \(V\) is a generic extension of \(\HOD\)
        for a forcing of size less than the least extendible cardinal,
        which implies the linearity of the Mitchell order on \(\kappa\)-complete
        Dodd sound ultrafilters in \(V\).
        The fact that the linearity of the Mitchell order on Dodd
        sound ultrafilters implies the linearity of the Mitchell order on
        normal fine ultrafilters is a result from \cite{MO}. The result only
        applies to normal fine ultrafilters on \(P_\textnormal{bd}(\lambda)\) if
        \(2^{<\lambda} = \lambda\), which is true by \cref{thm:gch}.
    \end{proof}
\end{thm}
\bibliographystyle{plain}
\bibliography{Bibliography.bib}
\end{document}